\documentclass{amsart}

\usepackage{amsmath,amssymb,amsfonts,amstext,amsthm,bm}
\usepackage{mathrsfs}
\usepackage{graphicx}
\usepackage{tikz}
\usepackage[all,cmtip]{xy}
\usepackage{xcolor}
\usepackage{enumitem}
\usepackage[pdfpagelabels,bookmarks,hyperindex,hyperfigures]{hyperref}
\usepackage[capitalize]{cleveref}

\newtheorem{corollary}{Corollary}[section]
\newtheorem{theorem}[corollary]{Theorem}

\newtheorem{lemma}[corollary]{Lemma}
\newtheorem{proposition}[corollary]{Proposition}

\newtheorem*{theorem*}{Theorem}

\theoremstyle{definition}

\newtheorem{definition}[corollary]{Definition}

\newtheorem{remark}[corollary]{Remark}

\newtheorem{example}[corollary]{Example}

\usepackage{faktor}

\newcommand{\trop}{\operatorname{trop}}
\newcommand{\Id}{\operatorname{Id}}
\newcommand{\fgId}{\operatorname{fgId}}
\newcommand{\Frac}{\operatorname{Frac}}
\newcommand{\Spec}{\operatorname{Spec}}
\newcommand{\maxSpec}{\operatorname{maxSpec}}

\renewcommand{\Vert}{\operatorname{Vert}}
\newcommand{\Supp}{\operatorname{Supp}}

\newcommand{\In}{\operatorname{in}}
\newcommand{\tr}{\operatorname{tr}}

\title[Initial degeneration]{Tropical initial degeneration for systems of algebraic differential equations}

\author[L. Bossinger et al.]{Lara Bossinger}
\address{Instituto de Matem\'aticas Unidad Oaxaca, 
Universidad Nacional Aut\'onoma de M\'exico, Oaxaca, Mexico.}
\urladdr{\url{https://www.matem.unam.mx/~lara}}

\author[]{Sebastian Falkensteiner}
\address{Max Planck Institute for Mathematics in the Sciences, Leipzig, Germany.}
\urladdr{\url{https://sites.google.com/view/falkensteiner-sebastian}}

\author[]{Cristhian Garay-L\'opez}
\address{Centro de Investigaci\'on en Matem\'aticas, A.C. (CIMAT),
Jalisco S/N, Col. Valenciana CP. 36023 Guanajuato, Gto, M\'exico.}
\urladdr{\url{http://personal.cimat.mx:8181/~cristhian.garay/}}

\author[]{Marc Paul Noordman}
\address{Bernoulli Institute, University of Groningen, The Netherlands}
\urladdr{\url{https://www.rug.nl/staff/m.p.noordman}}


\begin{document}
\maketitle

\begin{abstract}
We study the notion of degeneration for affine schemes associated to systems of algebraic differential equations with coefficients in the fraction field of a multivariate formal power series ring. 
In order to do this, we use an integral structure of this field that arises as the unit ball associated to the tropical valuation, first introduced in the context of tropical differential algebra. This unit ball turns out to be a particular type of integral domain, known as B\'ezout domain.
By applying to these systems a translation map along a vector of weights that emulates the one used in classical tropical algebraic geometry, the resulting translated systems will have coefficients in this unit ball. When the resulting quotient module over the unit ball is torsion-free, then it gives rise to integral models of the original system in which every prime ideal of the unit ball defines an initial degeneration, and they can be found as a base-change to the residue field of the prime ideal.

In particular, the closed fibres of our integral models can be rightfully called initial degenerations, since we show that the maximal ideals of this unit ball naturally correspond to monomial orders. We use this correspondence to define initial forms of differential polynomials and initial ideals of differential ideals, and we show that they share many features of their classical analogues.
\end{abstract}

\section{Introduction}

In algebraic geometry, the study of algebraic geometric objects has long shifted from studying a single object to studying families of objects with certain characteristics in common. 
More specifically, degeneration techniques are a basic tool used broadly in algebraic geometry and are the foundation for subjects such as \emph{the theory of integral models} for varieties, or \emph{deformation theory}.

The aim of the present article is to introduce these techniques to study the spaces of formal power series solutions  of systems of algebraic differential equations. 
In the algebraic case, when studying solutions to polynomial equations, Gr\"obner theory provides the necessary tools to construct an \emph{initial degeneration} 
of the polynomial system: 
these are flat families over the affine line $\pi:\mathcal X\to \mathbb A^1$ whose fibre $\mathcal X_t=\pi^{-1}(t)$ over (almost) any non zero value $t\in \mathbb A^1$ is (isomorphic to) the set of solutions of the original system, and whose fibre over the origin $\mathcal X_0=\pi^{-1}(0)$ is the set of solutions to a simplified polynomial system (for example, the polynomial equations are truncated to binomial or monomial equations).

Gr\"obner theory is closely related to tropical geometry as the fundamental theorem of tropical (algebraic) geometry demonstrates~\cite{maclagan2015introduction}. 
More precisely, every initial degeneration with irreducible special fibre corresponds to a point in the tropical variety (see, e.g.~\cite{Bos_initial}).
The generalization of the fundamental theorem to polynomial differential systems was initiated in~\cite{AGT16} (the ordinary case) and taken further to the partial case in~\cite{FGH20,mereta_fund}.
Our construction originates from these results, more precisely from the study of tropical solutions to tropical systems of algebraic differential equations, and our findings may be phrased in terms of non-Archimedean valuations or seminorms, as we explain below.

We would like to mention that this paper is not the first exploring possible generalizations of algebraic methods to the differential case. In particular, Gr\"obner theoretical methods have been employed in a similar setting in the works~\cite{ArocaRond,Aroca-Ilardi,FT20}.

\subsection{A new example of B\'ezout non-Archimedean norm}
We consider the fraction field $K(\!(\mathbf{t})\!)$ of the ring of formal power series in $m\geq1$ variables $\mathbf{t}=(t_1,\ldots,t_m)$ and with coefficients in a field of characteristic zero $K$.

We study the algebraic and geometric properties of a recently introduced non-Archimedean norm $\trop$
defined on $K(\!(\mathbf{t})\!)$ having values in the (idempotent) fraction semifield of the semiring of \emph{vertex polynomials} $V\mathbb{B}(\mathbf{t})$ (\cite{CGL,FGH20}, see \S\ref{Sect_valuationsFormalPowerSeriesRings} precisely \eqref{eq: def tropical seminorm}. 
This is the {tropical} seminorm, and it is not a classical Krull valuation, but rather a new example of the concept of B\'ezout $\ell$-valuation\footnote{Recall that an $\ell$-valuation is a seminorm $v:K\xrightarrow[]{}S$ defined on a field $K$ and taking values in an idempotent semifield; it differs from Krull valuations in the sense that $S$ does not need to be totally-ordered.}, which was introduced in \cite{RY}. 

Even if the target semifield $V\mathbb{B}(\mathbf{t})$ is not totally ordered, this setting behaves in a similar way to the Krull valuations; in particular, the subring $K(\!(\mathbf{t})\!)^\circ\subset K(\!(\mathbf{t})\!)$ of elements having $\trop$ norm bounded by $1\in V\mathbb{B}(\mathbf{t})$ turns out to be a B\'ezout domain. In this context, this subring is the analogue of the valuation ring of a Krull valuation, and we call it the {\it unit ball} of the seminorm.

This yields one of the first concrete applications of B\'ezout $\ell$-valuations through tropical geometry.  We choose to present this concept in the language of non-Archimedean seminorms using idempotent semiring theory, and we recall connections to B\'ezout $\ell$-valuations.

\subsection{Application}
Consider the differential ring $(K[\![\mathbf{t}]\!],D)$, where $D=\{\tfrac{\partial}{\partial t_i}\::\:i=1,\ldots,m\}$ is the set of usual derivations. The {tropical} seminorm $\trop$ appeared first in the tropical aspect of systems of differential algebraic equations with coefficients in $K[\![\mathbf{t}]\!]$, see \cite{FGH20}. 

The differential field $(K(\!(\mathbf{t})\!),D)$ is a differential extension of  $(K[\![\mathbf{t}]\!],D)$, and when $m>1$, the B\'ezout domain $K(\!(\mathbf{t})\!)^\circ\subset K(\!(\mathbf{t})\!)$ is a proper extension of $K[\![\mathbf{t}]\!]$. 
We apply these previous results to set up the theory of tropical initial degeneration of systems of algebraic differential equations with coefficients in $K(\!(\mathbf{t})\!)$.

\subsection{Initial degenerations}
\label{SS:ID}
Given an algebraic variety $X$ over a field $K$ and an integral domain $R$ with fraction field $\Frac(R)=K$, a \emph{model} of $X$ over the base $B=\Spec(R)$ is a flat morphism $\pi:\mathcal X \to B$ such that $X\cong \mathcal X\times_B\Spec(K)$.  

Denote by $F_{m,n}$, respectively $R_{m,n}$, the ring of  polynomials in the variables $\{x_{i,J}\::\:i=1,\ldots,n,\:J\in\mathbb{N}^m\}$ with coefficients in $K(\!(\mathbf{t})\!)$, respectively in the unit ball $K(\!(\mathbf{t})\!)^\circ$. 
For a given weight vector $w=(w_1,\ldots,w_n)\in \mathbb B[\![\mathbf{t}]\!]^n$ (here $\mathbb B$ denotes the Boolean semifield) and a differential polynomial $P\in F_{m,n}$, we define its {$w$-translation $P_w\in R_{m,n}$}  in~\eqref{eq:init form}. Broadly speaking we are generalizing the ordinary case which appeared in~\cite{FT20,HuGao2020}.
{The $w$-translated ideal} of a given  ideal in $F_{m,n}$, is the  ideal in $R_{m,n}$ generated by the   $w$-translations of all its elements (Definition~\ref{def:init ideal}).
An initial degeneration is specified by choosing a prime ideal $\mathfrak{p}\subset K(\!(\mathbf{t})\!)^\circ$.

\begin{theorem*}(Proposition~\ref{Proposition_Model} and Theorem~\ref{thm_degeneration})
Let $w\in \mathbb B[\![\mathbf{t}]\!]^n$ and $G \subset F_{m,n}$ be an ideal such that {the quotient $R_{m,n}/G_w$ of $R_{m,n}$ by the $w$-translated ideal $G_w$ is a torsion-free $K(\!(\mathbf{t})\!)^\circ$-module}. 
Then $\mathcal X(w):=\Spec(R_{m,n}/G_w)$ is a model for $\mathcal X(w)_{\eta}:=\Spec(F_{m,n}/G)$ over $K(\!(\mathbf{t})\!)^\circ$.
Moreover, given any maximal ideal $\mathfrak m\in \maxSpec( K(\!(\mathbf{t})\!)^\circ)$ the corresponding closed fibre satisfies
\[
\mathcal X(w)_{\mathfrak{m}}\cong \Spec\left(K\{x_{i,J}\}/\overline{G}_w\!\right) 
\]
{where $\overline{G}_w$ is the induced ideal in $K\{x_{i,J}\}$ (see the Proof of Lemma~\ref{lem:fiber_general} for the precise definition of $\overline{G}_w$).}
\end{theorem*}

A concept closely related to the fibre of a model is the one of initial ideals.
We use the $w$-translation map together with a maximal ideal $\mathfrak{m}\subset K(\!(\mathbf{t})\!)^{\circ}$ to define the initial form $\text{in}_{(w,\mathfrak{m})}(P)$ of $P\in F_{m,n}$ with respect to the pair $(w,\mathfrak{m})$, and the initial ideal $\text{in}_{(w,\mathfrak{m})}(G)$ with respect to the pair $(w,\mathfrak{m})$ of an ideal $G\subset F_{m,n}$.

We show that for a polynomial $P\in F_{m,n}$, our expressions \eqref{eq:explicit_translation} for its $w$-translation $P_w\in R_{m,n}$, and \eqref{eq:explicit_initial_deg} for its initial form $\text{in}_{(w,\mathfrak{m})}(P)$  are natural generalizations of the same constructions for the case of standard tropical algebra, c.f. \cite[\S5]{gubler2013guide}.
The second half of the above theorem follows from a careful analysis of the set of  maximal ideals in $K(\!(\mathbf{t})\!)^\circ$ denoted by $\maxSpec( K(\!(\mathbf{t})\!)^\circ)$. 
We find the following result which is proven in a  constructive manner in both directions.

\begin{theorem*}(\cref{thm_characterization})
There exists an explicit one-to-one correspondence between maximal ideals of $K(\!(\mathbf{t})\!)^{\circ}$ and monomial orders on $\mathbb{N}^m$.
\end{theorem*}

This characterization is important since it allows us to give explicit formulas for computing the initial form $\text{in}_{(w,\mathfrak{m})}(P)$ of $P\in F_{m,n}$ with respect to the pair $(w,\mathfrak{m})$, see for instance \eqref{eq:concrete_reduction}, \eqref{eq:explicit:fpwcoeff}.

\subsection{Outline}
In \cref{Sect_Preliminaries} we introduce non-Archimedean seminorms, idempotent semiring theory and B\'ezout $\ell$-valuations. We give results on these concepts which will be necessary for defining initial degeneration. Of uttermost importance hereby is the correspondence theorem between ideals in the domain and $k$-ideals in the image of a B\'ezout valuation (see \cref{ideals_are_the_same}) which is new up to our knowledge.
In Subsection~\ref{Sect_valuationsFormalPowerSeriesRings} we show that $\trop$ is indeed a B\'ezout seminorm on $K(\!(\mathbf{t})\!)$ (\cref{Proposition_trop_is_Bezout}) such that the previous result can be applied.
Moreover, in \cref{Proposition_nonNoetherian} we show that the ring $K(\!(\mathbf{t})\!)^{\circ}$ defined by $\trop$ is in the multivariate case not Noetherian.

In \cref{Sect_models_deg}, the notions from tropical differential algebra are recalled and results from \cref{Sect_Preliminaries} are used to show that models are found in this setting (\cref{Proposition_Model}) such that initial degenerations can be properly defined, see \cref{thm_degeneration}. We give the definitions of initial form and initial ideal at the pair $(w,\mathfrak{m})$ in Definition \ref{def:new_initials}.

The maximal ideals of $K(\!(\mathbf{t})\!)^{\circ}$ are studied detailed in the subsequent \cref{section_maximalIdeals}. In Proposition \ref{prop:multiplicative}, we show that taking the initial form at the pair $(w,\mathfrak{m})$ of a polynomial is a multiplicative map, as it is the case in classical tropical algebraic geometry. 

\section{Preliminaries}
\label{Sect_Preliminaries}
\subsection{Non-Archimedean seminorms}
Let $R$ be a (commutative) ring and let $S$ be an idempotent (commutative) semiring. On $S$ we define the  order $a \le b$ if and only if $a+b=b$. Order considerations on idempotent semirings will be made with respect to this order if no further remark is made.
\begin{definition}\label{def_gen_nav}
A {\em (non-Archimedean) seminorm} is a map $v:R\xrightarrow{}S$ from a ring $R$ to an idempotent semiring $S$ that satisfies
\begin{enumerate}
    \item (unit) $v(0)=0$ and $v(1)=1$;
    \item (sign) $v(-1)=1$;
    \item (submultiplicativity) $v(ab)\leq v(a)v(b)$; and
    \item (subaditivity) $v(a+b)\leq v(a)+v(b)$; 
\end{enumerate}
The seminorm $v$ is a {\em norm} if every $a \ne 0$ fulfills $v(a) \ne 0$, and it is called multiplicative if $v(ab)= v(a)v(b)$ holds for every $a,b \in R$. A multiplicative norm is called a {\em valuation}.
\end{definition}

Let $v:R\xrightarrow{}S$ be a seminorm. The set $R^{\circ}:=\{a \in R : v(a) \leq 1\}$ is called the \emph{unit ball of $v$}, and is a subring of $R$. 
By a little abuse of notation we also refer to the set $S^\circ:=\{x\in S:x\le 1\}$ as the unit ball of $S$; it is a subsemiring of $S$.

\begin{definition}
    A seminorm $v:R\xrightarrow{}S$ is called {\em integral} if $v(a) \leq 1$ for all $a \in R$.
\end{definition}

A seminorm $v:R\xrightarrow{}S$ is integral if and only if $R^\circ=R$.
By restricting the domain of definition of $v$ to $R^{\circ}$, an integral seminorm $v^{\circ}:R^\circ\to S^\circ$ is induced.

\begin{definition}
A subset $I$ of a semiring $S$ is an ideal if $0$, $a+b$ and $ac$ are elements of $I$ whenever $a,b\in I$ and $c\in S$. An ideal $I$ of $S$ is 
\begin{enumerate}
    \item a $k$-\emph{ideal} or a \emph{subtractive ideal}, if whenever $a+b\in I$, $a \in I$ and $b \in S$, then $b \in I$.
    \item \emph{prime}, if its complement $S\setminus I$ is a multiplicative subset of $S$.
\end{enumerate}
\end{definition}

\begin{lemma}\label{lem:substractive=downward closed}
Let $S$ be an idempotent semiring. An ideal $I$ of $S$ is \emph{downward closed} if whenever $b\in I$ and $a\le b$, we have $a\in I$. Subtractive ideals are equivalent to downward closed ideals.
\end{lemma}
\begin{proof}
Assume $I\subset S$ is subtractive and take $b\in I$ and $a\in S$ with $a\le b$. Then $a+b=b\in I$ which implies $a\in I$ as $I$ is subtractive. 

To the contrary, if $I$ is downward closed, consider $b\in I, a\in S$ and $a+b\in I$. 
Since $a+(a+b) = a+b$, it follows that $a \le a+b$, and because $I$ is downward closed, $a \in I$.
\end{proof}

Let $R$ be a ring. Let $\Id(R)$, respectively $\fgId(R)$, denote the set of ideals of $R$, respectively finitely generated ideals of $R$. 
Note that $\Id(R)$ and $\fgId(R)$ are semirings with respect to the sum and product of ideals.

If $S$ is a semiring, we denote by $\Id_k(S)$ the set of $k$-ideals of $S$. By \cite[Corollary 3.8]{BG} the map $u_R:R\to \fgId(R)$ sending an element $a\in R$ to its principal ideal $(a)$ is a surjective valuation that induces an isomorphism of semirings $\Id(R)\cong \Id_k(\fgId(R))$. 
We use this correspondence for describing the maximal ideals of the unit balls $R^{\circ}$ and $V\mathbb{B}(\mathbf{t})^\circ$ of the tropical valuation $\trop:K(\!(\mathbf{t})\!)\rightarrow V\mathbb{B}(\mathbf{t})$ in terms of monomial orders in Section~\ref{section_maximalIdeals}.

\subsection{B\'ezout $\ell$-valuations as seminorms}
If $S$ is a semiring, we denote by $U(S)\subset (S,\times,1)$ its group of multiplicatively invertible elements.

We continue with the following concepts.

\begin{definition}
{An (integral) domain} $R$ is called a \emph{B\'ezout domain} if every finitely generated ideal is principal.
\end{definition}

Let $R$ be a B\'ezout domain. The greatest common divisor (gcd) and least common multiple (lcm) of two elements $a,b\in R$ always exist, and they satisfy $$ab=u\,\text{gcd}(a,b)\,\text{lcm}(a,b)$$ for some $u\in U(R)$.

We write $\Gamma(R):=R/U(R)$ and denote by $\pi:R\to \Gamma(R)$ the quotient projection, sending $a\in R$ to its class $\pi(a)=[a]$. Note that $\Gamma(R)$ endowed with the product $[a][b]=[ab]$ and addition $[a]+[b]=[\text{gcd}(a,b)]$ is an idempotent semiring, and $\pi:R\to \Gamma(R)$ becomes an integral valuation.

\begin{remark}
If $R$ is a B\'ezout domain, then there is an isomorphism of semirings $\fgId(R)\cong \Gamma(R)$, thus the norm $u_R:R\to \fgId(R)$ coincides with the quotient projection $\pi:R\to \Gamma(R)$.
\end{remark}

If $R$ is a B\'ezout domain, then the semiring $\Gamma(R)$ is multiplicatively cancellative and the fraction semifield $\Frac(\Gamma(R))$ exists. The valuation $\pi:R\to \Gamma(R)$ can be extended to a valuation $\Frac(\pi):\Frac(R)\to \Frac(\Gamma(R))$ with the usual definition: $\Frac(\pi)(\tfrac{a}{b})=\tfrac{\pi(a)}{\pi(b)}$.

\begin{definition}
Let $R$ be a B\'ezout domain. Its divisibility semiring is the semiring $\Gamma(R)$, and its divisibility semifield is the semifield $\Frac(\Gamma(R))$.
\end{definition} 

Throughout the remaining part of section, $R$ denotes a field and $S$ a semifield.
In this case, $\Frac(R) = R$, $\Gamma(R)=\{ [0],[1] \}$, and $\pi:R\to \Gamma(R)$ is the trivial norm.

\begin{definition}
\label{def:Bezout_snorm}
Let $R$ be a field and let $S$ be a semifield. 
We say that a multiplicative seminorm $v:R\to S$ is \emph{B\'ezout} if for every $a,b\in R$, there exists $x,y\in R^\circ$ such that $v(xa+yb)=v(a)+v(b)$.
\end{definition}

Since $R$ is a field, it follows that a B\'ezout seminorm is automatically a norm, hence, it is a valuation after Definition \ref{def_gen_nav}.

Let $v:R\to S$ be B\'ezout valuation with induced integral norm $v^{\circ}:R^{\circ}\to S^{\circ}$. We start with the following observation about principal ideals of $R^{\circ}$.

\begin{lemma}
\label{ideals_are_downwards_closed_gen}
Given $a\in R^\circ$, we have  
\[ (a) = \{b \in R^\circ : v(b) \leq v(a)\}. \]
\end{lemma}
\begin{proof}
The statement is clear if $a = 0$, so assume that $a \neq 0$. Then a straightforward computation shows that
\begin{align*}
    v(b) \leq v(a)  \Longleftrightarrow v\left(\frac{b}{a}\right)=\frac{v(b)}{v(a)}\leq1 \Longleftrightarrow \frac{b}{a} \in R^\circ \Longleftrightarrow b \in (a).\quad\qedhere
\end{align*} 
\end{proof}

This already has the consequence that ideals are determined by seminorms, in the following sense.

\begin{corollary}\label{trop_determines_elements_of_ideal}
Let $I$ be an ideal of $R^\circ$ and $a \in I$. Then for any $b \in R^\circ$ with $v(b) \leq v(a)$, we have $b \in I$.
\end{corollary}
\begin{proof}
If $v(b) \leq v(a)$ then by \cref{ideals_are_downwards_closed_gen}, we have $b \in (a) \subset I$.
\end{proof}

Therefore, an ideal $I \subset R^\circ$ is uniquely determined by the subset $v(I) \subseteq S^\circ$ and we have
\begin{equation}\label{eq:I and v(I)}
I = \{a \in R  : v(a) \in v(I)\}.    
\end{equation}

It follows from \cref{ideals_are_downwards_closed_gen} that $v(I)$ is a {downward closed} subset of $S^\circ$. 
In particular, $v(I)$ is closed under multiplication with elements in $S^\circ$: if $b \in v(I)$ and $a \in S^\circ$, then $ab \leq b$ (since $a \leq 1$) and so $ab \in v(I)$.

\begin{lemma}\label{closed_under_addition_gen}
Let $v: R \to S$ be a B\'ezout valuation such that $v^{\circ}$ is surjective. 
Let $I$ be an ideal of $R^\circ$ and let $x, y \in v(I)$.
Then $x + y \in v(I)$.
\end{lemma}
\begin{proof}
As $v^\circ$ is surjective there exist $f,g\in R^\circ$ such that $v(f)=x$ and $v(g)=y$.
Moreover, by~\eqref{eq:I and v(I)} we have that $f,g\in I$.
As $v$ is B\'ezout there exist $\alpha,\beta\in R^\circ$ such that 
\[
v(\alpha f+\beta g)=v(f)+v(g)=x+y.
\]
Notice that $\alpha f+\beta g\in I $ as $I$ is an ideal and then again by~\eqref{eq:I and v(I)} we conclude $x+y\in v(I)$. 
\end{proof}

\cref{closed_under_addition_gen} implies that $v(I)$ is closed under summation, hence it is subtractive by Lemma~\ref{lem:substractive=downward closed}. 
Thus, if $I\subset R^{\circ}$ is an ideal, then $v(I)\subset S^\circ$ is a $k$-ideal.

\begin{theorem}[Correspondence theorem]\label{ideals_are_the_same}
Let $v : R \to S$ be a B\'ezout valuation such that $v^\circ:R^\circ\to S^\circ$ is surjective. 
There is a one-to-one correspondence between ideals $I\subset R^{\circ}$ and $k$-ideals $J \subset S^\circ$ given by $I=v^{-1}(J)$ and $J=v(I)$. This correspondence preserves prime ideals.
\end{theorem}
\begin{proof}
The only thing left to verify is that $v^{-1}(J)$ is an ideal of $R^{\circ}$.
Let $f,g \in v^{-1}(J)=\{ g \in R^{\circ} : v(g) \in J\}$ and $r \in R^{\circ}$. Then $f+g \in v^{-1}(J)$ and $rf \in v^{-1}(J)$, since $v(f+g) \le v(f)+v(g) \in J$ and $v(rf) \le v(r)v(f) \in J$ and $J$ is downwards closed.

For the second statement, let $I$ and $J$ be two ideals such that $I=v^{-1}(J)$ and $v(I)=J$. Since $f\notin I$ if and only if $v(f)\notin J$ and $v: R^{\circ} \to S^{\circ}$ is surjective, the statement follows.
\end{proof}

\begin{remark}
As we pointed out before, the correspondence Theorem~\ref{ideals_are_the_same} is related to~\cite[Corollary 3.8]{BG}, the difference is that in our case, the direct image $v(I)\subset S^\circ$ of an ideal $I\subset R^\circ$ under the surjective valuation $v^\circ $ is already a $k$-ideal of $S^\circ$, and there is no need to take its $k$-closure.
\end{remark}

Although we will not need a generalization of the correspondence theorem to radical and primary ideals, we will give a proof here for the sake of completeness.

\begin{proposition}
\label{prop:preservation}
The correspondence induced by a B\'ezout valuation $v$ preserves radical and primary ideals.
\end{proposition}
\begin{proof}
First, let us show the preservation of radical ideals. 
Any radical ideal can be written as intersection of minimal prime ideals, see e.g. \cite[Corollary 2.12]{Eisenbud_commutative}. 
Since $v$ preserves primes, it remains to show that minimal primes are preserved. 
For this purpose, let $I$ be a minimal prime ideal in $R^{\circ}$, i.e., for every prime ideal $J$ with $J \subseteq I$ it holds that $J=I$, and let $J' \subseteq v(I)$ be a prime ideal. Then there is a prime ideal $J$ such that $v(J)=J'$. Thus, $v(J) \subseteq v(I)$ and therefore $J \subseteq I$. 
By the minimality, $J=I$. 
The same argument can be used for $v^{-1}$ and minimal prime $k$-ideals in $S^{\circ}$ and the statement follows.

Now, let us show the preservation of primary ideals. 
Let $J\subset S^{\circ}$ be a primary $k$-ideal and let $I=v^{-1}(J)$. If $fg\in I$, then $v(fg)=v(f)v(g)\in J$. For $v(f)\notin J$, we have $f \notin I$. Then $v(g)^n=v(g^n)\in J$ and $g^n\in I$ for some $n \in \mathbb{N}$.

Conversely, let $I\subset R^{\circ}$ be a primary ideal and let $J=v(I)$. Let $xy\in J$ and $f \in v^{-1}(x), g \in v^{-1}(y)$ be arbitrary. Since $fg \in v^{-1}(xy) \subset I$, then $f \in I$ or $g^n \in I$ for some natural number $n$. 
Thus, $v(f) \in J$ or $v(g^n)=v(g)^n \in J$.
\end{proof}

The next result shows that for a surjective B\'ezout valuation $v:R\to S$, the unit ball $R^\circ$ is B\'ezout and $v$ is characterized by the norm $\pi:R^\circ\to \Gamma(R^\circ).$

\begin{lemma}
\label{cor:structure}
Let $v : R \to S$ be a surjective B\'ezout valuation. Then $v$ is isomorphic to the norm $\Frac(\pi):\Frac(R^\circ)\to \Frac(\Gamma(R^\circ))$.
\end{lemma}
\begin{proof}
First we show that $S=\Frac(S^\circ)$. Given some $x \in S$ we want to write it as $x = a/b$ with $a, b \leq 1$. For this we just set $a = x/(1+x)$ and $b = 1/(1+x)$. Since $x \leq 1 + x$ and $1 \leq 1 + x$ we have $a, b \leq 1$ and also clearly $x = a/b$. So $S^\circ\subseteq S \subseteq \Frac(S^\circ)$, and the inclusion $\Frac(S^\circ) \subseteq S$ follows from the fact that $S$ is a semifield.

Assume that $v$ is surjective. Then for $x\in S^\circ$ there exists $a\in R$ such that $v(a)=x$, but $a\in R^{\circ}$ by definition. Thus, $v^\circ$ is surjective.

Since $\Frac(R^\circ)=R$~\cite[Proposition 1]{RY}, there are $a,b\in R^\circ$ such that $x=\tfrac{a}{b}$. So $v(x)=\tfrac{v(a)}{v(b)}=\tfrac{v^\circ(a)}{v^\circ(b)}$, and $v:R\to S$ is the extension $\Frac(v^\circ)$ of $v^{\circ}:R^{\circ}\to S^{\circ}$.

If $v:R\to S$ is a B\'ezout valuation, then $R^\circ$ is a B\'ezout domain by~\cite[Proposition 1]{RY}. Then every finitely generated ideal of $R^{\circ}$ is principal, and the isomorphism $\Id(R^\circ)\cong \Id_k(S^\circ)$ of Theorem \ref{ideals_are_the_same} restricts to a isomorphism $\Gamma(R^\circ)\cong\fgId(R^\circ)\cong \fgId_k(S^\circ)$ sending $(a)$ to $v^{\circ}((a))$. Note that $v^{\circ}((a))=(v^{\circ}(a))_k$: the $k$-closure of the principal ideal generated by $v^{\circ}(a)$.

We now make the identification $S^\circ\cong \fgId_k(S^\circ)$ by sending $s$ to $(s)_k$. 
This is clearly surjective, and it is injective since $(s)_k=(t)_k$ implies $s\leq t$ and $t\leq s$, so $t=s$.
Thus $\phi:\Gamma(R^\circ) \to S^\circ$ given by $\phi((a))=v^{\circ}(a)$ is an isomorphism satisfying $v^\circ=\phi\circ\pi$, and $v=\Frac(v^\circ)=\Frac(\phi)\circ \Frac(\pi)$.
\end{proof}

\begin{corollary}
\label{cor:surjectivity}
Let $v : R \to S$ be a B\'ezout valuation. Then $v$ is surjective if and only if $v^\circ$ is surjective.
\end{corollary}
\begin{proof}
Let $v^\circ$ be surjective. 
Since $S=\Frac(S^\circ)$ (see proof of \cref{cor:structure}), for any $y\in S$ there are $y_1,y_2\in S^\circ$ such that $y=\tfrac{y_1}{y_2}$, and since $v^\circ$ is surjective, there exists $a_1,a_2\in R$ such that $v(\tfrac{a_1}{a_2})=\tfrac{v^\circ(a_1)}{v^\circ(a_2)}=\tfrac{y_1}{y_2}=y$. Thus, $v$ is surjective.

The reverse direction is proven in \cref{cor:structure}.
\end{proof}

\begin{corollary}
\label{cor:unit_ball}
Let $v : R \to S$ be a surjective B\'ezout valuation. Then $U(R^\circ)=\{x\in R\::\: v(x)=1\}$.
\end{corollary}
\begin{proof}
If $x\in R$ satisfies $v(x)=1$, then $x\in R^\circ$. Since $v$ is surjective, we have that $\Frac(R^\circ)=R$, so there exist $y,z\in R^\circ$ such that $x=\tfrac{y}{z}\neq0$, then  $v(y)=v(z)$, and $x^{-1}=\tfrac{z}{y}\in R^\circ$.

Conversely, if  $x\in U(R^\circ)$, then $v(x)$ is invertible in $S^\circ$ (since $v$ is multiplicative), but $U(S^\circ)=\{1\}$  since it is simple (and thus sharp).
\end{proof}

In the next section we apply these results to a concrete example of a B\'ezout valuation.

\subsection{Valuations in the context of formal power series rings}
\label{Sect_valuationsFormalPowerSeriesRings}
Let $\mathbb B=\{0,1\}$ denote the Boolean idempotent semifield. Fix an integer $m\ge 1$ and a tuple of variables $\mathbf{t}=(t_1,\dots,t_m)$. We denote by $\mathbb{B}[\![\mathbf{t}]\!]$ the semiring of formal Boolean power series, and by 
$V\mathbb{B}[\mathbf{t}]$ the idempotent semiring of vertex polynomials as in~\cite{CGL}. 

\begin{remark}
Recall that $\mathbb{B}[\![\mathbf{t}]\!]$ is isomorphic to $\mathcal{P}(\mathbb{N}^m)$ and $V\mathbb{B}[\mathbf{t}]$ is isomorphic to $\mathbb{T}_m$ from~\cite{FGH20} with the isomorphism given by taking supports.
\end{remark}

More precisely, the elements of $V\mathbb{B}[\mathbf{t}]$ are subsets of $\mathbb N^m$ that are equal to the vertices of the Newton polyhedra they generate. The addition $\oplus$ on $V\mathbb{B}[\mathbf{t}]$ is given by taking the set union of two vertex polynomials and then projecting onto the vertices of the outcome. Similarly, the product $\odot$ on $V\mathbb{B}[\mathbf{t}]$ is given by taking the Minkowski sum of the vertex polynomials and then projecting onto the vertices of the outcome. 

Since $V\mathbb{B}[\mathbf{t}]$ is integral (i.e. multiplicatively cancellative, that is whenever $a\odot b=a\odot c$ then either $a=0$ or $b=c$), we can construct the fraction semifield $V\mathbb{B}(\mathbf{t}):=\text{Frac}(V\mathbb{B}[\mathbf{t}])$ as follows: 
the elements of $V\mathbb{B}(\mathbf{t})$ are of the form $\frac{a}{b}$ with $a,b\in V\mathbb{B}[\mathbf{t}]$ and $b\neq0$, the map $V\mathbb{B}[\mathbf{t}]\xrightarrow[]{} V\mathbb{B}(\mathbf{t})$ sending $a$ to $\frac{a}{1}$ is an embedding, and $\frac{a}{b}=\frac{c}{d}$ if and only if $a \odot d=b \odot c$. 
The sum and product of fractions are defined as usual: $\frac{a}{b}\odot\frac{c}{d}=\frac{a\odot c}{b\odot d}$ and $\frac{a}{b} \oplus \frac{c}{d}=\frac{a \odot d\oplus b\odot c}{b\odot d}$. Note that on $V\mathbb{B}(\mathbf{t})$ we have the order $\frac{a}{b}\leq\frac{c}{d}$ if and only if $a\odot d\leq b\odot c$.

Let $K[\![\mathbf{t}]\!]$ be the ring of formal power series in the variables $\mathbf{t}=(t_1,\ldots,t_m)$ with coefficients in the field $K$. We define the {tropical seminorm}
\begin{equation}\label{eq: def tropical seminorm}
\trop:K[\![\mathbf{t}]\!]\rightarrow V\mathbb{B}[\mathbf{t}]    
\end{equation}
as the composition of the maps $\Supp: K[\![\mathbf{t}]\!]\rightarrow \mathbb{B}[\![\mathbf{t}]\!]$ given by taking the support set of a power series, composed with
$V:\mathbb{B}[\![\mathbf{t}]\!]\rightarrow V\mathbb{B}[\mathbf{t}]$ given by projecting onto the vertex set of the Newton polyhedron generated by an element of $\mathbb{B}[\![\mathbf{t}]\!]$. The tropical seminorm is a surjective  valuation, so we will call it the tropical valuation.

Recall that the map $\Frac(\trop):K(\!(\mathbf{t})\!)\rightarrow V\mathbb{B}(\mathbf{t})$ is defined by $\Frac(\trop)(\frac{\varphi}{\psi}):=\frac{\trop(\varphi)}{\trop(\psi)}$. This map is also a surjective valuation by \cite[Corollary 7.3]{CGL}, thus we will call it also the tropical valuation and will  be denoted by $\trop$. Moreover, we have:

\begin{proposition}
\label{Proposition_trop_is_Bezout}
The tropical valuation $\trop: K(\!(\mathbf{t})\!)\to V\mathbb{B}(\mathbf{t})$  is a surjective $K$-algebra B\'ezout valuation.
\end{proposition}

\begin{proof}
Let $\varphi=\sum_Ia_I\mathbf{t}^I$ and $\psi=\sum_Ib_I\mathbf{t}^I$ be nonzero elements in $ K[\![\mathbf{t}]\!]$. If $\trop(\varphi+\psi)\neq\trop(\varphi)\oplus\trop(\psi)$, there exists $I\in \trop(\varphi)\oplus\trop(\psi)$ such that $a_I+b_I=0$.  Since $\trop(\varphi)\oplus\trop(\psi)$ is a polynomial, we can choose $M\in\mathbb{N}$ such that $a_I+Mb_I\neq0$ for all $I\in \trop(\varphi)\oplus\trop(\psi)$. It follows that $\trop(\varphi+M\mathbf{t}^0\psi)=\trop(\varphi)\oplus\trop(\psi)$ and $\trop(M\mathbf{t}^0)=1$.

Now consider $\frac{\varphi_1}{\varphi_2}$ and $\frac{\psi_1}{\psi_2}$ elements in $K(\!(\mathbf{t})\!)^*$. We apply the above argument to $\varphi=\varphi_1\psi_2$ and $\psi=\varphi_2\psi_1$ to find 

\begin{equation*}
\trop\left(\frac{\varphi_1}{\varphi_2}+M\mathbf{t}^0\frac{\psi_1}{\psi_2}\right)
=\frac{\trop(\varphi+M\mathbf{t}^0\psi)}{\trop(\varphi_2\psi_2)}
=\trop\left(\frac{\varphi_1}{\varphi_2}\right)\oplus\trop \left(\frac{\psi_1}{\psi_2}\right),
\end{equation*}
which finishes the proof.
\end{proof}

The following result gives a concrete characterization of the tropical valuation $\trop:K(\!(\mathbf{t})\!)\to V\mathbb{B}(\mathbf{t})$.

\begin{corollary}
Let $\trop^\circ:K(\!(\mathbf{t})\!)^\circ\to V\mathbb{B}(\mathbf{t})^\circ$ be the induced integral valuation. Then $V\mathbb{B}(\mathbf{t})^\circ\cong \{\trop(x)\leq 1\}/\{\trop(x)=1\}$, and $\trop^\circ$ is isomorphic to the resulting quotient projection.
\end{corollary}
\begin{proof}
Since $\trop: K(\!(\mathbf{t})\!)\to V\mathbb{B}(\mathbf{t})$ is a surjective B\'ezout seminorm, by \cref{Proposition_trop_is_Bezout} and \cref{cor:structure}, it follows that 
\[
V\mathbb{B}(\mathbf{t})^\circ\cong\Gamma(K(\!(\mathbf{t})\!)^\circ)=K(\!(\mathbf{t})\!)^\circ/U(K(\!(\mathbf{t})\!)^\circ),
\]
and $U(K(\!(\mathbf{t})\!)^\circ)=\{x\in K(\!(\mathbf{t})\!)\::\:\trop(x)=1\}$ by Corollary~\ref{cor:unit_ball}.
\end{proof}

The correspondence Theorem \ref{ideals_are_the_same} says that we have a semiring isomorphism  
\begin{equation}
    \Id(K(\!(\mathbf{t})\!)^\circ)\cong \Id_k(V\mathbb{B}(\mathbf{t})^\circ)\cong \Id_k(K(\!(\mathbf{t})\!)^\circ/\{\trop(x)=1\}).
\end{equation}

We use this correspondence to characterize the maximal $k$-ideals of the semiring $V\mathbb{B}(\mathbf{t})^\circ$ in Corollary \ref{cor:max:ideals:unitball}.

\begin{example}
\label{ex:case_m=1.1}
Suppose that $m=1$. Since for $f \in K[\![t]\!]$ it holds that $\trop(f)=\min(\Supp(f))$, we have that
\begin{align*}
K(\!(t)\!)^\circ &= \{ f/g \in K(\!(t)\!) : \min(\min(\Supp(f)),\min(\Supp(g))) = \min(\Supp(g)) \} \\ &= \{ f/g \in K(\!(t)\!) : \min(\Supp(f)) \ge \min(\Supp(g)) \} =K[\![t]\!],
\end{align*}
which is Noetherian.
\end{example}

\begin{proposition}\label{Proposition_nonNoetherian}
For $m>1$, the ring $K(\!(\mathbf{t})\!)^\circ$ is not Noetherian.
\end{proposition}
\begin{proof}
We demonstrate the statement in the case of $\mathbf{t}=(t,u)$, the general case follows.
Define, for each $n \geq 1$, the element
\[ 
\omega_n := 
\frac{t^{2n+1}+u^{2n+1}}{t^{2n+1}+t^nu^n+u^{2n+1}} \in V\mathbb{B}(t,u). 
\]
Then for each $n$ we have that $\omega_n \in V\mathbb{B}(t,u)^\circ$, and also  $\omega_n < \omega_{n+1}$. Define
\[ 
I_n = \{q \in K(\!(\mathbf{t})\!)^\circ : \trop(q) \leq \omega_n\}. 
\]
It follows from  the correspondence Theorem \ref{ideals_are_the_same} that $I_n$ defines an ideal in $K(\!(\mathbf{t})\!)^{\circ}$,   {which can also be proven directly}. 
Then it follows that $I_n \subsetneq I_{n+1}$ and the $I_n$ form a strictly increasing sequence of ideals in $K(\!(\mathbf{t})\!)^\circ$. 
\end{proof}

\section{Models over the unit ball and initial degenerations}
\label{Sect_models_deg}
For a fixed integer $m\geq1$, we consider the B\'ezout-valued field $(K(\!(\mathbf{t})\!),\trop)$ together with its unit ball $K(\!(\mathbf{t})\!)^\circ\subset K(\!(\mathbf{t})\!)$ from the previous section. 
The purpose of this section is to create models over $K(\!(\mathbf{t})\!)^\circ$ of certain schemes $X$ {defined over} $K(\!(\mathbf{t})\!)$,
as well as initial degenerations of certain $K(\!(\mathbf{t})\!)$-algebras which appear in the theory of differential algebra. In this first section, we follow \cite[\S4]{gubler2013guide}.

\subsection{Models over the unit ball}
{Let $R$ be an integral domain such that $\Frac(R)=K$. 
An $R$-\emph{model} of a scheme $X$ over the field $K$ is a flat scheme $\mathcal X$ over $R$ with generic fiber $\mathcal{X}_\eta=\mathcal{X}_K=X$. See \cite[Definition 4.1]{gubler2013guide}.

Equivalently, if $B=\Spec(R)$, then} a model of $X$ over $B$ is a flat  morphism\footnote{Here \emph{flat} is in the sense of Definition II-28 and \emph{proper} as defined on page 95 in \cite{EH_GeomSch}.} $\pi:\mathcal X \to B$ such that $X= \mathcal X\times_B\Spec(K)$. 
Then $X$ is called the \emph{generic fibre} of $\mathcal X$. {Sometimes we also ask the map $\pi:\mathcal X \to B$ to be proper.}

\begin{lemma}
\label{lem:flat}
    A module over $K(\!(\mathbf{t})\!)^\circ$ is flat if and only if it is torsion-free.
\end{lemma}
\begin{proof}
   Any flat module is torsion-free.  {The converse follows from the fact that torsion-free modules over Pr\"ufer domains, and hence B\'ezout domains, are flat (see e.g.~\cite[Theorem 3.3]{bazzoni2006prufer}).}
\end{proof}

{For another fixed integer $n\geq 1$, we consider the set of variables $\{x_{i,J}\::\:1\leq i\leq n,\:J\in\mathbb{N}^m\}$, which we abbreviate simply by $\{x_{i,J}\}$.
Let us denote by $R_{m,n}$ the ring of polynomials in the variables $\{x_{i,J}\}$ with coefficients in the unit ball $K(\!(\mathbf{t})\!)^\circ$. It is customary to express the elements of $R_{m,n}$ as finite sums $P=\sum_Ma_ME_M$ where $a_M\in K(\!(\mathbf{t})\!)^\circ$ and $E_M$ are finite products of the variables $\{x_{i,J}\}$.

We consider flat schemes over $K(\!(\mathbf{t})\!)^\circ$ of the form $\mathcal{X}=Spec(A)$ for $A=R_{m,n}/J$, thus its generic fiber $X=\mathcal{X}_\eta$ is of the form $X=Spec(A_{K(\!(\mathbf{t})\!)})$ for $A_{K(\!(\mathbf{t})\!)}:=A\otimes_{K(\!(\mathbf{t})\!)^{\circ}}K(\!(\mathbf{t})\!)$. 
By flatness we have $A\subset A_{K(\!(\mathbf{t})\!)}$. 
A closed subscheme $Y$ of $X$ is given by an ideal $I_Y\subset A_{K(\!(\mathbf{t})\!)}$. 
The closure $\overline{Y}$ of $Y$ in $\mathcal{X}$ (in the Zariski topology) is defined as the closed subscheme of $\mathcal{X}$ given by the ideal $I_Y\cap A$.}
{The following result and its proof are straightforward generalizations of \cite[Proposition 4.4]{gubler2013guide}.

\begin{proposition}
\label{prop:creation_of_models}
    The closure $\overline{Y}$ of $Y$ in $\mathcal{X}$ is the unique closed subscheme of $\mathcal{X}$ with generic fiber $Y$ which is flat over $K(\!(\mathbf{t})\!)^\circ$
\end{proposition}
}

{In particular, we can apply Proposition \ref{prop:creation_of_models} to $A=R_{m,n}$, so that $A_{K(\!(\mathbf{t})\!)}=F_{m,n}$ is the ring of polynomials in the variables $\{x_{i,J}\}$ and coefficients in $K(\!(\mathbf{t})\!)$. Thus, given an ideal $I_X\subset F_{m,n}$, its contraction $I_X\cap R_{m,n}$   gives already a model $\mathcal{X}:=\Spec(R_{m,n}/I_X\cap R_{m,n})$  over $K(\!(\mathbf{t})\!)^\circ$ for $X=\Spec(F_{m,n}/I_X)$. This shows that we can construct models effectively by taking the  closure of ideals $I\subset F_{m,n}$.}

We will introduce the motivation from differential algebra behind these particular choices in \cref{section:TDA}.

We denote by $\kappa(\mathfrak{p})$ the residue field of a point $\mathfrak{p}\in \Spec(K(\!(\mathbf{t})\!) ^\circ)$. If $\mathcal{X}=\Spec(R_{m,n}/J)$ is a flat scheme over $K(\!(\mathbf{t})\!) ^\circ$, the fibre $\mathcal{X}_{\mathfrak{p}}$ of $\mathcal{X}$ over $\mathfrak{p}$ is the spectrum of the ring $(R_{m,n}/J)\otimes_{K(\!({\bf t})\!)^\circ}\kappa(\mathfrak{p})$. If the prime ideal under consideration is a maximal ideal $\mathfrak{m}\subset K(\!(\mathbf{t})\!)^{\circ}$, then $\kappa(\mathfrak{m})=K(\!({\bf t})\!)^\circ/\mathfrak{m}$.

Let $\mathfrak{m}\subset K(\!(\mathbf{t})\!)^{\circ}$ be a maximal ideal and consider the quotient projection $\pi:K(\!({\bf t})\!)^\circ\xrightarrow[]{}K(\!({\bf t})\!)^\circ/\mathfrak{m}$. 
By Corollary~\ref{quotient_by_maximal_ideal_is_K} below, there exists an isomorphism $\psi_{\mathfrak m}:K(\!({\bf t})\!)^\circ/\mathfrak{m}\to K$, so there is an induced morphism $\psi_{\mathfrak m}\circ\pi:K(\!({\bf t})\!)^\circ\xrightarrow[]{}K$ (which with the help of Theorem \ref{thm_characterization} can be described concretely as in \eqref{eq:concrete_reduction}).

This in turn induces a morphism 
\begin{equation}
\label{eq:reduction}
    \psi_{\mathfrak m}\circ\pi:R_{m,n}\xrightarrow[]{}K[x_{i,J}]
\end{equation}
by applying $\psi_{\mathfrak m}\circ\pi$ coefficient-wise: if $P=\sum_Ma_ME_M$ is an element in $R_{m,n}$, then $\psi_{\mathfrak m}\circ\pi(P)=\sum_M\psi_{\mathfrak m}\circ\pi(a_M)E_M$.

\begin{lemma}
\label{lem:fiber_general}
Let  $J\subset R_{m,n}$ be an ideal and let $\mathfrak{m}\subset K(\!(\mathbf{t})\!)^{\circ}$ be a maximal ideal. Then $(R_{m,n}/J)\otimes_{K(\!({\bf t})\!)^\circ} (K(\!({\bf t})\!)^\circ/\mathfrak m) \cong K[x_{i,J}]/\overline{J}_{\mathfrak{m}}$, where $\overline{J}_{\mathfrak{m}}\subset K[x_{i,J}]$ denotes the extended ideal under the morphism \eqref{eq:reduction}.  
\end{lemma}
\begin{proof}
    By little abuse of notation we denote by $\mathfrak m$ also the ideal it generates in $R_{m,n}$ and also the induced map $\psi_{\mathfrak m}:R_{m,n}/\mathfrak{m}\to K[x_{i,J}]$.
    Then 
    \begin{eqnarray*}
    (R_{m,n}/J)\otimes_{K(\!({\bf t})\!)^\circ} (K(\!({\bf t})\!)^\circ/\mathfrak m) &\cong&
    (R_{m,n}/J)/(R_{m,n}/J \cdot \mathfrak m) \\
    &\cong& (R_{m,n}/\mathfrak m)/(J \cdot \mathfrak m) \\
    &\cong& K[x_{i,J}]/\overline{J}_{\mathfrak{m}}
    \end{eqnarray*}
    where $\overline{J}_{\mathfrak{m}}$ is the ideal generated by $\psi_{\mathfrak m}(J\cdot \mathfrak m)\subset K[x_{i,J}]$.
\end{proof}

In particular, if $\mathcal{X}=\Spec(R_{m,n}/J)$ is a flat scheme over $K(\!(\mathbf{t})\!) ^\circ$ and $\mathfrak{m}\subset K(\!(\mathbf{t})\!)^{\circ}$ is a maximal ideal, then the closed fiber $\mathcal{X}_{\mathfrak{m}}$ of the model $\mathcal{X}$ is simply $\Spec(K[x_{i,J}]/\overline{J}_{\mathfrak{m}})$, where $\overline{J}_{\mathfrak{m}}$ is the ideal in $K[x_{i,J}]$ generated by $\{\psi_{\mathfrak m}\circ\pi(P)\::\:P\in J\}$.

\subsection{Tropical differential algebra}\label{section:TDA}
Let $K$ be a field of characteristic zero and let $n,m\geq1$ be integers. Recall that $F_{m,n}$ denotes the ring of polynomials in the variables $\{x_{i,J}\}$ and coefficients in $K(\!(\mathbf{t})\!)$.

In order to define our notion of (initial) degeneration in this setting, we need some concepts from tropical differential algebra, for which we follow~\cite[\S7.1]{CGL}.

{We denote by $D=\{\tfrac{\partial}{\partial t_i}\::\:i=1,\ldots,n\}$ the usual partial derivations defined on $K[\![\mathbf{t}]\!]$. If $E$ is a monomial in the variables $\{x_{i,J}\}$ and $\varphi=(\varphi_1,\ldots,\varphi_n)\in K[\![\mathbf{t}]\!]^n$, we define the differential evaluation $E(\varphi)$ by replacing the variable  $x_{i,J}$ for $J=(j_1,\ldots,j_m)$ with $\tfrac{\partial^{j_1+\cdots +j_m}\varphi_i}{\partial t_1^{j_1} \cdots \partial t_m^{j_m}}$. If $P=\sum_M\frac{a_M}{b_M}E_M$ is an element in $F_{m,n}$,  where $a_M,b_M \in K[\![\mathbf{t}]\!]$ and $E_M$ are differential monomials, we define $\text{ev}_P(\varphi)=\sum_M\frac{a_M}{b_M}E_M(\varphi)$. 

Denote by $e_K:\mathbb{B}[\![\mathbf{t}]\!]\to K[\![\mathbf{t}]\!]$ the section of the map $\Supp: K[\![\mathbf{t}]\!]\rightarrow \mathbb{B}[\![\mathbf{t}]\!]$ that sends $A\subset\mathbb{N}^m$ to the formal power series $e_K(A)=\sum_{I\in  A}\mathbf{t}^I$. If $E$ is a monomial in the variables $\{x_{i,J}\}$ and $w=(w_1,\ldots,w_n)\in\mathbb{B}[\![\mathbf{t}]\!]^n$, then we have 
\begin{equation}
\label{eq:ev_mon}
    V(E(w))=\trop(E(e_K(w))),
\end{equation}
where $e_K(w)=(e_K(w_1),\ldots,e_K(w_n))\in K[\![\mathbf{t}]\!]^n$. 

The next step is to use these tools from differential algebra together with the tropical seminorm $\trop: K(\!(\mathbf{t})\!)\to V\mathbb{B}(\mathbf{t})$ to define a seminorm $\trop_w:F_{m,n}\longrightarrow V\mathbb{B}(\mathbf{t})$ that depends on a fixed weight vector $w=(w_1,\ldots,w_n)\in\mathbb{B}[\![\mathbf{t}]\!]^n$.}

\begin{definition}\label{tropical_valuation}
Given a fixed weight vector $w\in\mathbb{B}[\![\mathbf{t}]\!]^n$, we define the map $\trop_w:F_{m,n}\longrightarrow V\mathbb{B}(\mathbf{t})$ by sending $P=\sum_M\frac{a_M}{b_M}E_M$ to
\begin{equation}
\label{eq:tropw}
    \trop_w(P):=\bigoplus_M\trop\biggl(\frac{a_M}{b_M}E_M(e_K(w))\biggr)=\bigoplus_M\biggl(\frac{\trop(a_M)}{\trop(b_M)}\odot V(E_M(w))\biggr).
\end{equation}
\end{definition}

The equality between the two expressions appearing in \eqref{eq:tropw} follows from \eqref{eq:ev_mon}. It was shown in~\cite{CGL} that $\trop_w$ is a $K$-algebra seminorm.

Given a weight vector $w=(w_1,\ldots,w_n)\in \mathbb{B}[\![\mathbf{t}]\!]^n$ we will now use the seminorm \eqref{eq:tropw} to construct a $w$-translation map 
\begin{equation}
\label{eq:wtr}
    \tr_w:F_{m,n}\xrightarrow[]{}R_{m,n}.
\end{equation}

The following construction is inspired by the works~\cite{FT20,HuGao2020} in the ordinary case. We will also see in \eqref{eq:explicit_translation} that the explicit expression for the map \eqref{eq:wtr} does not fall very far to the usual expression of the translation to the origin of a torus, which is operative in classical tropical geometry, see~\cite[\S5]{gubler2013guide}.

{Let us go back to the differential evaluation $E(e_K(w))$ which appears in the definition of~\eqref{eq:tropw}. Recall that for $\varphi\in K[\![\mathbf{t}]\!]\setminus\{0\}$, we denote by $\overline{\varphi}\in K[\![\mathbf{t}]\!]$ the restriction of $\varphi$ to the vertices of its Newton polygon. 
Then $\varphi=\overline{\varphi}+\widetilde{\varphi}$, where the {\it initial form}  of $\varphi$ is $\overline{\varphi}$, and it is a polynomial.

For every variable $x_{i,J}$ appearing in $E$, we write $\Theta(J)(e_K(w_i))=\frac{\partial^{j_1+\cdots +j_m}e_K(w_i)}{\partial t_1^{j_1} \cdots \partial t_m^{j_m}}$. 
Thus the differential evaluation $E(e_K(w))$ equals the usual algebraic evaluation of $E$ at the vector $(\Theta(J)(e_K(w_i))\::\:i,J)$, and we have a decomposition 
\begin{equation}
\label{eq:initial_data_dm}
    E(e_K(w))=E(\overline{\Theta(J)(e_K(w_i)})\::\:i,J)+R,
\end{equation}
where again the {\it initial form} of $E(e_K(w))$ is stored in the first term of the right hand side of~\eqref{eq:initial_data_dm}.
}

If $P\in F_{m,n}$ satisfies $\trop_w(P)=\frac{a}{b}\neq0$ with  {$A=\trop(a)$ and $B=\trop(b)$}, we set \[T(\trop_w(P))^{-1}:= {\frac{B}{A}} \in K(\!(\mathbf{t})\!)\]
where $A,B$ are considered in $K(\!(\mathbf{t})\!)$ via the natural embedding. Since $A$ and $B$ are uniquely determined by $a$ and $b$, respectively, $T(\trop_w(P))^{-1}$ exists. Moreover, $T(\trop_w(P))^{-1}$ is well-defined by taking into account the multiplicativity of $\trop$.

We now specify the polynomial $\tr_w(P)=P_w$ by substituting every instance of $x_{i,J}$ in $P$ by $\overline{\Theta(J)(e_K(w_i))}x_{i,J}$  and then multiplying the result by $T(\trop_w(P))^{-1}$:
\begin{equation}\label{eq:init form}
P_w:=\begin{cases}
T(\trop_w(P))^{-1}P(\overline{\Theta(J)(e_K(w_i))}x_{i,J}),&\text{ if }\trop_w(P)\neq0\\
0,&\text{ if }\trop_w(P)=0. \\
\end{cases}
\end{equation}

\begin{proposition}
If $P=\sum_Ma_ME_M\in F_{m,n}$, then $P_w$ is an element in $R_{m,n}$. Also, if $\trop_w(P)\neq0$, then 
\begin{equation}
\label{eq:explicit_translation}
    P_w=\sum_M[T(\trop_w(P))^{-1}a_ME_M(\overline{\Theta(J)(e_K(w_i)}))]E_M,
\end{equation}
where $E_M(\overline{\Theta(J)(e_K(w_i)}))$ denotes the usual algebraic evaluation of the monomial $E_M$ at the vector $(\overline{\Theta(J)(e_K(w_i)})\:i,J)$.
\end{proposition}
\begin{proof}
It was shown in~\cite[Lemma 7.14]{CGL} that \eqref{eq:init form} is an element in $R_{m,n}$ when instead of using $(\overline{\Theta(J)(e_K(w_i))}\::\:i,J)$ one uses $(T(w_i,J)\::\:i,J)$, where 
$T(w_i,J)=e_K(\Vert((w_i-J)_{\geq0}))$. These two expressions are linked as follows 

\begin{equation*}
    T(w_i,J)=e_K(\Supp(\overline{\Theta(J)(e_K(w_i)}))),\quad \text{for all }i,J.
\end{equation*}

Thus the first part follows since $T(w_i,J)$ and  $\overline{\Theta(J)(e_K(w_i))}$ have the same support for all $i,J$ and non-negative integer coefficients.

The second part follows from \eqref{eq:initial_data_dm}, since replacing every instance of $x_{i,J}$ in $P$ by $\overline{\Theta(J)(e_K(w_i))}x_{i,J}$ sends the monomial $E_M$ of $P$ to $E_M(\overline{\Theta(J)(e_K(w_i)}))E_M$.
\end{proof}

 {
 \begin{remark}
 \label{rem:two_versions}
     Note that the following polynomial 
 \begin{equation*}
    \widetilde{P_w}=\sum_M[T(\trop_w(P))^{-1}a_ME_M(e_K(w))]E_M,
\end{equation*}
is a modified version of \eqref{eq:explicit_translation}, and it represents  the exact copy in this setting of the translation to the origin of a torus appearing in \cite[\S5]{gubler2013guide}. 
However, our definition is simpler since it extracts only the initial data (which is a polynomial) of the evaluations $E_M(e_K(w))$ (which are series), as we showed above.  
 \end{remark}
 }

\begin{definition}\label{def:init ideal}
Let $w=(w_1,\ldots,w_n)\in\mathbb{B}[\![\mathbf{t}]\!]^n$. 
For a given ideal $G\subset F_{m,n}$, we define its {\it $w$-translation}  by $G_w$ as the ideal in $R_{m,n}$ generated by $\{P_w\::\:P\in G\}$.
\end{definition}

\begin{proposition}\label{Proposition_Model}
Let $w=(w_1,\ldots,w_n)\in\mathbb{B}[\![\mathbf{t}]\!]^n$ and $G\subset F_{m,n}$ be an ideal. If {$R_{m,n}/G_w$ is torsion-free}, then 
\begin{equation*}
    \mathcal{X}(w):=\Spec(R_{m,n}/G_w)
\end{equation*}
is a model over $K(\!(\mathbf{t})\!) ^\circ$ for the generic fibre $\mathcal{X}(w)_\eta$.
\end{proposition}

\begin{proof}
{This is a consequence of Lemma \ref{lem:flat}.}
\end{proof}

Let $w\in\mathbb{B}[\![\mathbf{t}]\!]^n$ and $G\subset F_{m,n}$ be such that {$R_{m,n}/G_w$ is torsion-free}. The fibre $\mathcal{X}(w)_{\mathfrak{p}}$ of $\mathcal{X}(w)$ over $\mathfrak{p}\in \Spec(K(\!(\mathbf{t})\!) ^\circ)$ is the spectrum of the ring $(R_{m,n}/G_w)\otimes_{K(\!({\bf t})\!)^\circ}\kappa(\mathfrak{p})$.

\begin{definition}
Let $w=(w_1,\ldots,w_n)\in\mathbb{B}[\![\mathbf{t}]\!]^n$ and $G\subset F_{m,n}$ be such that {$R_{m,n}/G_w$ is torsion-free}. 
Further let $\mathcal{X}(w)=\Spec(R_{m,n}/G_w)$ with generic fibre $X=\mathcal{X}(w)_\eta$. 
Given a point $\mathfrak{p}\in \Spec(K(\!(\mathbf{t})\!) ^\circ)\setminus\{\eta\}$ the fibre $\mathcal{X}(w)_{\mathfrak{p}}$ of $\mathcal{X}(w)$ over $\mathfrak{p}\in \Spec(K(\!(\mathbf{t})\!) ^\circ)$ is called the \emph{initial degeneration} of $X$ at the pair $(w,\mathfrak{p})$. 
\end{definition}

\begin{theorem}\label{thm_degeneration}
Let $w=(w_1,\ldots,w_n)\in\mathbb{B}[\![\mathbf{t}]\!]^n$ and $G\subset F_{m,n}$ be an ideal such that {$R_{m,n}/G_w$ is torsion-free}. 
Then for any maximal ideal $\mathfrak{m}\subset K(\!(\mathbf{t})\!)^{\circ}$ we have that
\[
\mathcal X(w)_{\mathfrak{m}}\cong \Spec\left(K[x_{i,J}]/\overline{G}_w\!\right) 
\]
where $\overline{G}_w\subset K[x_{i,J}]$ denotes the extended ideal under the morphism \eqref{eq:reduction}.

\end{theorem}
\begin{proof}
We have that $\mathcal X(w)_{\mathfrak m}$ is the spectrum of $(R_{m,n}/G_w)\otimes_{K(\!({\bf t})\!)^\circ} (K(\!({\bf t})\!)^\circ/\mathfrak m)$, which is isomorphic to $K[x_{i,J}]/(\overline{G_w})_{\mathfrak{m}}$ by Lemma \ref{lem:fiber_general}.
\end{proof}

\begin{remark}
In the case of $m=1$ (i.e. the case of discrete valuation rings) since $K[\![t]\!]= K(\!(t)\!)^{\circ}\subset K(\!(t)\!)$ it follows that $\Spec(K(\!(t)\!)^{\circ})=\{(0),\mathfrak{m}=(t)\}$.
So there is only one initial degeneration $\mathcal{X}(w)_\mathfrak{m}$ for {every} $w=(w_1,\ldots,w_n)\in\mathbb{B}[\![t]\!]^n$.
As the residue field of $K(\!(t)\!)^{\circ}$ is $\kappa(\mathfrak{m})=K$ we have that $\mathcal{X}(w)_\mathfrak{m}$ is a scheme over $K$.

We will study the maximal spectrum of $K(\!(\mathbf{t})\!)^{\circ}$ for the general case of $m>1$ in \S\ref{section_maximalIdeals}.
\end{remark}

\subsection{Initial ideals along maximal ideals}
{
Let $G\subset F_{m,n}$ be any ideal and $w\in\mathbb{B}[\![\mathbf{t}]\!]^n$, so that $G_w$ is an ideal of $ R_{m,n}$. If $\mathfrak{m}\subset K(\!(\mathbf{t})\!)^{\circ}$ is a maximal ideal, in this section we will focus  on the extended ideal $(\overline{G_w})_{\mathfrak{m}}\subset K[x_{i,J}]$ under the morphism \eqref{eq:reduction}, no matter if  $ R_{m,n}/G_w$ is flat or not. We will show that these extended ideals share many properties with the  construction of initial ideals in classical tropical geometry.
We start with the following result 
\begin{proposition}
\label{prop:in:deg}
Let $G\subset F_{m,n}$ be an ideal, $w=(w_1,\ldots,w_n)\in\mathbb{B}[\![\mathbf{t}]\!]^n$ and $\mathfrak{m}\subset K(\!(\mathbf{t})\!)^{\circ}$ a maximal ideal. Then the extended ideal $(\overline{G_w})_{\mathfrak{m}}\subset K[x_{i,J}]$ under the morphism \eqref{eq:reduction}  is the ideal in $K[x_{i,J}]$ generated by $\{\psi_{\mathfrak m}\circ\pi(P_w)\::\:P\in G\}$.

\end{proposition}
\begin{proof}
    If $P=\sum_Ma_ME_M\in F_{m,n}$ and $\trop_w(P)\neq0$, we have from  \eqref{eq:explicit_translation} that $P_w=\sum_M[T(\trop_w(P))^{-1}a_ME_M(\overline{\Theta(J)(e_K(w_i)}))]E_M$, thus 
\begin{equation}
\label{eq:explicit_initial_deg}
    \psi_{\mathfrak m}\circ\pi(P_w)=\sum_M\psi_{\mathfrak m}\circ\pi[T(\trop_w(P))^{-1}a_ME_M(\overline{\Theta(J)(e_K(w_i)}))]E_M.
\end{equation}

Now, it is clear that the ideal in $K[x_{i,J}]$ generated by $\{\psi_{\mathfrak m}\circ\pi(P_w)\::\:P\in G\}$ is contained in $(\overline{G_w})_{\mathfrak{m}}$. Conversely, an element of $G_w$ is a finite sum $R=\sum_P Q_PP_w$ with $P\in G$ and $Q_P\in R_{m,n}$, thus every generator  $(\overline{G_w})_{\mathfrak{m}}$ is of the form 
$$
\psi_{\mathfrak m}\circ\pi(R)=\psi_{\mathfrak m}\circ\pi\bigl(\sum_P Q_PP_w\bigr)=\sum_P \psi_{\mathfrak m}\circ\pi(Q_P)\psi_{\mathfrak m}\circ\pi(P_w),
$$
since the map $\psi_{\mathfrak m}\circ\pi$ from \eqref{eq:reduction} is a homomorphism of rings.
\end{proof}

\begin{remark}
    \label{rem:two_versions_initial}
    Note that the following polynomial 
 \begin{equation*}
    \psi_{\mathfrak m}\circ\pi(    \widetilde{P_w})=\sum_M\psi_{\mathfrak m}\circ\pi[T(\trop_w(P))^{-1}a_ME_M(e_K(w))]E_M,
\end{equation*}
is a modified version of \eqref{eq:explicit_initial_deg}, and it would be the exact copy in this setting of the initial form appearing in \cite[Remark 5.7]{gubler2013guide}, which is the usual definition of initial form which is operative in tropical geometry. 
Once again, our definition is simpler for the same reasons described in Remark \ref{rem:two_versions}.
\end{remark}

\begin{definition}
\label{def:new_initials}
For a pair $(w,\mathfrak{m})$ of a weight $w=(w_1,\ldots,w_n)\in\mathbb{B}[\![\mathbf{t}]\!]^n$ and a maximal ideal $\mathfrak{m}\subset K(\!(\mathbf{t})\!)^{\circ}$, we denote \eqref{eq:explicit_initial_deg} by $\text{in}_{(w,\mathfrak{m})}(P):=\psi_{\mathfrak m}\circ\pi(P_w)$, and call it the \emph{initial form of $P\in F_{m,n}$ at $(w,\mathfrak{m})$}.

Similarly,  if $G\subset F_{m,n}$ is an ideal, we denote  by $\text{in}_{(w,\mathfrak{m})}(G)$ the ideal in $K[x_{i,J}]$ generated by $\{\text{in}_{(w,\mathfrak{m})}(P)\::\:P\in G\}$, and call it the \emph{initial ideal of $G$ at $(w,\mathfrak{m})$}.
\end{definition}
 }

{We now return to the differential setting. The set $D=\{\tfrac{\partial}{\partial t_i}\::\:i=1,\ldots,n\}$ of partial derivations defined on $K[\![\mathbf{t}]\!]$ can be extended to $K(\!(\mathbf{t})\!)$ in the usual way, and also to $F_{m,n}$. If we denote them also by $D$, the pair $(F_{m,n},D)$ is a differential ring, and an ideal $G\subset F_{m,n}$ is differential if it is closed under the action of $D$.

We are particularly interested in the case in which $G\subset F_{m,n}$ is a differential ideal and $m>1$.}

\begin{example}\label{exp}
Let $P=x_{(1,1)}-tx_{(0,0)} \in K(\!(t,u)\!)\{x\}=F_{2,1}$ and $w = \mathbb{N}^2 \setminus \{(1,1)\}\in\mathbb{B}[\![t,u]\!]$. 
Then $\trop_w(P)=\{(1,0),(0,1)\}$ and
\[
P_w = x_{(1,1)} - \frac{t}{t+u}x_{(0,0)}. 
\]
Consider a maximal ideal $\mathfrak{m}\subset K(\!(t,u)\!)^{\circ}$. {We  see from Example \ref{exp_continuation} that}
\[
P_w\!\mod \mathfrak{m}=\left\{\begin{matrix} x_{(1,1)}, & \frac{t}{t+u}\in\mathfrak m\\
x_{(1,1)}- x_{(0,0)}, & \frac{t}{t+u}\notin \mathfrak m\end{matrix}\right.
\]
The derivatives of $P$ are of the form
\[P_{(j_1,j_2)}= x_{(1+j_1,1+j_2)} - tx_{(j_1,j_2)} - j_1x_{(j_1-1,j_2)} .\]
We obtain for every $(j_1,j_2) \ne (0,0)$ that $\trop_w(P_{(j_1,j_2)})=\{(0,0)\}$ and
\[ (P_{(j_1,j_2)})_w =
\begin{cases}
x_{(1+j_1,1+j_2)} - t(t+u)x_{(j_1,j_2)} - j_1x_{(j_1-1,j_2)}, & (j_1,j_2)=(1,1) \\
x_{(1+j_1,1+j_2)} - tx_{(j_1,j_2)} - j_1(t+u)x_{(j_1-1,j_2)}, & (j_1,j_2)=(2,1) \\
x_{(1+j_1,1+j_2)} - tx_{(j_1,j_2)} - j_1x_{(j_1-1,j_2)}, & \text{otherwise.}
\end{cases} \]
Since the differential ideal generated by $P$ in $F_{2,1}$, denoted as $[P]$, is prime, the algebraic ideal generated by the $(P_{(j_1,j_2)})_w$ already gives us $[P]_w$. 
Moreover, there is no element in $[P]$ that is in $K(\!(t,u)\!)[x_{(0,0)},x_{(1,0)},x_{(0,1)}]$ 
and since all differential monomials of $(P_{(j_1,j_2)})_w$ are the same as that of $P_{(j_1,j_2)}$, there is also no such element in $[P]_w$. 
All differential monomials involving higher derivatives can be reduced to one of order one or zero such that we obtain
\[ R_{2,1}/[P]_w \cong K(\!(t,u)\!)^{\circ}[x_{(0,0)},x_{(1,0)},x_{(0,1)}].\]
A characterization of prime ideals of such a polynomial ring is given e.g. in~\cite{ferrero1997prime}.
\end{example}

\section{Maximal ideals of the unit ball}\label{section_maximalIdeals}

In \S\ref{Sect_models_deg} we saw that for any point $\mathfrak{p}\in \Spec(K(\!(\mathbf{t})\!) ^\circ)\setminus\{\eta\}$, we can construct a degeneration of a generic fibre. Among the most important prime ideals are the maximal ideals, and in this section we characterize the maximal ideals of the B\'ezout domain $K(\!(\mathbf{t})\!)^\circ$.
We fix an integer $m\geq1$. First we recall the definition of monomial order.

\begin{definition}\label{def:mono order} 
Let $<$ be a total order on the monoid $(\mathbb{N}^m,+,0)$. Then $<$ is a \emph{monomial order} if 
\begin{itemize}
    \item[(a)] $0<a$ for all $a\in \mathbb{N}^m\setminus \{0\}$,
    \item[(b)] $a<b$ implies $a+c<b+c$ for all $c\in \mathbb{N}^m$.
\end{itemize}
\end{definition}

\begin{lemma} \label{lem:mono is vertex}
Every monomial order on $\mathbb{N}^m$ satisfies that for any non-empty subset $S \subset \mathbb{N}^m$ there is a unique minimum $\min_<(S)$ which moreover is a vertex of {the convex hull of} $S$. In particular, a monomial order is a well-order.
\end{lemma}

\begin{proof}
Consider $a\in \mathbb{N}^m$ and observe that, by Definition~\ref{def:mono order}(b), $a=\min_<\{a+\mathbb{N}^m\}.$
For a subset $\emptyset\neq S\subset \mathbb{N}^m$, this observation generalizes as follows. 
Let $P(S)$ be the convex hull of $S+\mathbb{N}^m$ in $\mathbb R^m$ and let $U(S)$ be the union of all bounded faces of $P(S)$.
Then $U(S)\cap \mathbb{N}^m$ contains the vertex set of $S$ (which agrees with the set of vertices of $P(S)$) and moreover 
\[
\min{}_<\{U(S)\cap \mathbb{N}^m\} = \min{}_<\{S\}.
\]
We need to show that this minimum is achieved at a unique vertex of $S$.
As $U(S)\cap \mathbb{N}^m$ is finite, the monomial order $<$ can be represented by a weight vector $w\in \mathbb R^m$, i.e. $\min_<\{U(S)\cap \mathbb{N}^m\}=\min_w\{U(S)\cap \mathbb{N}^m\}$ for some $w\in \mathbb R^m$. 
Let $m_1<\dots<m_k$ be the elements of $U(S)\cap \mathbb{N}^m$. 
Then any $w\in \mathbb R^m$ satisfying $w\cdot m_1<w\cdot m_2<\dots<w\cdot m_k$ represents $<$;
moreover, we may assume $w\in \mathbb{N}^m_{> 0}$ (see e.g. Lemma 3.1.1 in~\cite{HerzogHibi}).

{Suppose $m=\min_<\{U(S)\cap \mathbb{N}^m\}$ is not a vertex. 
Then there exist $0<\lambda_i<1$ with $\sum_{i=1}^k \lambda_i = 1$ such that $m=\sum_{i=1}^k \lambda_i m_i$. In particular, 
\[
w\cdot m= \sum_{i=1}^k \lambda_i\, w\cdot m_i.
\]
So either $\lambda_{1}=\cdots=\lambda_{k}$ and $w\cdot m= w\cdot m_1 = \cdots = w \cdot m_k$, but this contradicts $w\cdot m_1<\dots<w\cdot m_k$; or $w\cdot m_i<w\cdot m<w\cdot m_j$ for some $1 \le i,j \le k$ which contradicts the assumption that $m=\min_w\{U(S)\cap \mathbb{N}^m\}$.
}
\end{proof}

\begin{remark}
Lemma~\ref{lem:mono is vertex} is closely related to~\cite[Lemma 3.7]{ArocaRond}, yet slightly different: Aroca and Rond are dealing with total orders $<$ in $\mathbb{R}^m$ compatible with its group structure, and which can be seen as monomial orders on $\mathbb R^m$. 
In loc. cit. it is shown that for any such order $<$ and any rational polyhedral cone $\sigma\subset \mathbb R^m$ which is {non-negative, i.e., $0 \le a$ for all $a \in \sigma$, the set $\sigma\cap \mathbb Z^m$ is well ordered.}
\end{remark}

The aim of this section is to establish the following bijection between the set of monomial orders on the monoid $(\mathbb{N}^m,+,0)$ and the set of maximal ideals of $K(\!(\mathbf{t})\!)^\circ$.

\begin{theorem}
\label{thm_characterization}
There is an identification between maximal ideals $\mathfrak{m}$ of $K(\!(\mathbf{t})\!)^\circ$ and monomial orders on $\mathbb{N}^m$ defined, for $I,J \in \mathbb{N}^m$, as
\begin{equation}
\label{eq_characterization}
I < J \quad \Longleftrightarrow \quad \frac{\mathbf{t}^J}{\mathbf{t}^I + \mathbf{t}^J} \in \mathfrak{m} \quad \Longleftrightarrow \quad \frac{\mathbf{t}^I}{\mathbf{t}^I + \mathbf{t}^J} \notin \mathfrak{m}.
\end{equation}
\end{theorem}

We will use the remaining of this section to prove this result; the respective identifications are given in \cref{prop_maxideal} and \cref{Proposition_MaxIdeal_gives_Order}.

Given $f\in K[\![\mathbf{t}]\!]$ and $I\in\mathbb{N}^m$, we may denote by $f_I$ the coefficient of $\mathbf{t}^I$ in $f$.
\begin{proposition}\label{prop_maxideal}
Given a monomial order $<$ on $\mathbb N^m$, the set 
\[ 
\mathfrak{m}_{<} := \{f/g \in K(\!(\mathbf{t})\!)^\circ : f_{\min\!{}_{<} ( \Supp(g) )}=0\}
\]
is a maximal ideal of $K(\!(\mathbf{t})\!)^\circ$.
\end{proposition}
\begin{proof}
We claim that $\mathfrak{m}_{<}$ is the kernel of the map

\begin{equation}
    \begin{aligned}
        c_< : K(\!(\mathbf{t})\!)^\circ &\to K\\
        \frac{f}{g} &\longmapsto \frac{f_{\min\!{}_{<} ( \Supp(g) )}}{g_{\min\!{}_{<} ( \Supp(g) )}}.
    \end{aligned}
\end{equation}

Let us first show that the map $c_<$ is well-defined. Let $f/g=h/\ell$ with $f,g,h,\ell \in K[\![\mathbf{t}]\!]$, we need to show that $$f_{\min\!{}_{<} ( \Supp(g) )}\ell_{\min\!{}_{<} ( \Supp(\ell))}=h_{\min\!{}_{<} ( \Supp(\ell) )}g_{\min\!{}_{<} ( \Supp(g) )}.$$
If $f_{\min\!{}_{<}(\Supp(g))}=0$ or $h_{\min\!{}_{<}(\Supp(\ell))}=0$, the equality follows easily. Let us suppose that $f_{\min\!{}_{<}(\Supp(g))}\neq0$ and $h_{\min\!{}_{<}(\Supp(\ell))}\neq0$. 
Note that $\min\!{}_{<}(\Supp(g))\leq \min\!{}_{<}(\Supp(f))$ and $\min\!{}_{<}(\Supp(g))< \min\!{}_{<}(\Supp(f))$ if and only if $f_{\min\!{}_{<}(\Supp(g))}=0$. 
Thus, $\min\!{}_{<}(\Supp(f))=\min\!{}_{<}(\Supp(g))$ and $\min\!{}_{<}(\Supp(h))=\min\!{}_{<}(\Supp(\ell))$ for the same reason.
Since $f\ell=gh$, we have $$X:=\trop(f\ell)=\trop(f)\trop(\ell)=\trop(hg)=\trop(h)\trop(g),$$
and $I\in X$ if and only if there exist unique $J,K,L,M$ in the respective vertex sets such that $I=J+K=L+M$. In particular, for every $I\in X$ we have $$(f\ell)_I=f_J\ell_K=g_Lh_M=(gh)_I.$$
We set $$I_< := \min\!{}_{<}(\Supp(f))+\min\!{}_{<}(\Supp(\ell))=\min\!{}_{<}(\Supp(g))+\min\!{}_{<}(\Supp(h)).$$
We need show that $I_<\in X$. Suppose that $I_<=A+B$ for some $A\in \trop(f)$ and $B\in\trop(\ell)$, then $\min\!{}_{<}(\Supp({f}))\leq A$ and $\min\!{}_{<}(\Supp(\ell))\leq B$. Thus, 
$$I_<=\min\!{}_{<}(\Supp({f}\ell))=\min\!{}_{<}(\Supp({f}))+\min\!{}_{<}(\Supp(\ell))\leq A+B=I_<,$$
which holds only for {the uniquely determined} $A=\min\!{}_{<}(\Supp({f}))$ and $B=\min\!{}_{<}(\Supp(\ell))$. Thus, $I_<\in X$.

The additivity and multiplicativity of $c_<$ can be shown in a similar way. 
Thus, $c_<$ defines a surjective homomorphism such that, by the first isomorphism theorem,
\[K(\!(\mathbf{t})\!)^{\circ}/\mathfrak{m}_{<} = K(\!(\mathbf{t})\!)^{\circ}/\text{ker}(c_<) \cong K,\]
which is the case exactly if $\mathfrak{m}_{<}$ is a maximal ideal.
\end{proof}

Note that given a monomial order $<$ on $\mathbb N^m$, expressions of the form $\frac{\mathbf{t}^J}{\mathbf{t}^I + \mathbf{t}^J}$ are in $\mathfrak{m}_<$ if and only if $I<J$. 
Later we will use exactly those fractions to define a monomial order on $\mathbb N^m$ from a given maximal ideal $\mathfrak{m}\subset K(\!(\mathbf{t})\!)^\circ$.
Before we do that, we first establish that $K(\!(\mathbf{t})\!)^\circ / \mathfrak{m} \cong K$ for every maximal ideal $\mathfrak{m}$ of $K(\!(\mathbf{t})\!)^\circ$.

\begin{definition}
\label{dfn_relevant}
Let $a\leq b$ in $V\mathbb{B}[\mathbf{t}]$, with supports $A$,$B$ respectively. We say that $a$ is \emph{irrelevant} for $b$ if  $A\cap B=\emptyset$, and we write $a\ll b$. Otherwise we say that $a$ is relevant for $b$. If $\frac{a}{b}\leq \frac{c}{d}$ in $V\mathbb{B}(\mathbf{t})$, then $\frac{a}{b}\ll \frac{c}{d}$ if and only if $a\odot d\ll c\odot b$ in $V\mathbb{B}[\mathbf{t}]$.
\end{definition}

\begin{lemma}\label{irrelevant_contained_in_maximal_ideal}
Let $q \in K(\!(\mathbf{t})\!)^\circ$ be such that $\trop(q) \ll 1$. Then $q$ is contained in every maximal ideal of $K(\!(\mathbf{t})\!)^\circ$. 
\end{lemma}

\begin{proof}
Let $\mathfrak{m}$ be a maximal ideal of $K(\!(\mathbf{t})\!)^\circ$. Suppose it does not contain $q$. 
Then $(1) = \mathfrak{m} + (q)$ by maximality of $\mathfrak{m}$. 
So there are $r \in \mathfrak{m}$ and $s \in K(\!(\mathbf{t})\!)^\circ$ such that $1 = r + qs$. 
Since $\trop(q) \ll 1$ and $\trop(s) \leq 1$ we have $\trop(qs) \ll 1$. 
Now write $r = f/g$ and $qs = h/g$ for $f,g,h \in K[\![\mathbf{t}]\!]$ (we have already equalized the denominators here). 
Then $1 = r+qs = (f + h)/g$. 
But $\trop(h/g) \ll 1$ means that the support of $h$ has no vertices of $g$. 
So $f$ has in its support all the vertices of $g$, meaning that $\trop(f) = \trop(g)$, so that $r = f/g$ is a unit in $K(\!(\mathbf{t})\!)^\circ$. 
But $r \in \mathfrak{m}$, so $\mathfrak{m} = (1)$, which is not a maximal ideal of $K(\!(\mathbf{t})\!)^\circ$. 
\end{proof}

\begin{lemma}\label{product_of_differences}
Given $q \in K(\!(\mathbf{t})\!)^\circ$, there exists an integer $n \geq 1$ and elements $\alpha_1, \ldots, \alpha_n \in K$ such that $\trop\left(\prod_{k = 1}^n ( q - \alpha_k) \right)  \ll 1.$
\end{lemma}

\begin{proof}
Let $q = f/g$ with $f,g \in K[\![\mathbf{t}]\!]$. 
If $f\ll g$, then $(q-0)=q\ll1$. 
Otherwise, if $\{I_1\ldots, I_n\}$ is the vertex set of $g$, then $\trop(f)\cap \{I_1\ldots, I_n\}\neq\emptyset$. 
For each $k = 1, \ldots, n$, we set 
$\alpha_k := \frac{f_{I_k}}{g_{I_k}}$.
Notice that the denominator $g_{I_k}$ is never zero, but the numerator, and thus $\alpha_k$, might be zero. 
Let $r = \prod_{k = 1}^n (q - \alpha_k)$. We claim that $\trop(r) \ll 1$. 
First, we compute 
\[ 
r = \prod_{k = 1}^n (q - \alpha_k) = \frac{\prod_{k=1}^n [f - \alpha_k g]}{g^n}. 
\]
We need to show that $\trop(\prod_{k=1}^n [f - \alpha_k g]) \ll \trop(g^n)$. 
Note that the vertices of $g^n$ are of the form $n I_k$ for $k = 1,\ldots, n$.
Fix a value for $k$; we prove that $nI_k$ is not a vertex of $(f - \alpha_1 g)\cdots (f - \alpha_n g)$:
If $I_k\notin \trop(f)\cap \{I_1\ldots, I_n\}$, then $\alpha_k=0$ and $(f-\alpha_kg)_{I_k}=f_{I_k}=0$.
If $I_k\in \trop(f)\cap \{I_1\ldots, I_n\}$, then $(f-\alpha_kg)_{I_k}=0$.

Pick any weight vector $w \in \mathbb R^m$ such that $I_k$ is the $w$-minimal vertex of $\Supp(g)$.
Since $\trop(f) \leq \trop(g)$ we have that $\trop(f - \alpha_i g) \leq \trop(g)$ for each $i$, and so either $I_k$ appears in $\trop(f - \alpha_i g)$ in which case it is the $w$-minimal vertex of it, or $I_k$ does not appear, in which case the $w$-minimal vertex of $f - \alpha_i g$ is strictly bigger than $I_k$. 
The latter happens at least once, namely for $i = k$, because by construction, the coefficient $(f - \alpha_k g)_{I_k}$ is zero. 
The $w$-minimal vertex of the product $\prod_i (f_i - \alpha_i g_i)$ is therefore strictly bigger (with respect to the $w$-ordering) than $nI_k$. 
Since $nI_k$ is the $w$-minimal vertex of $\Supp(g^n)$, it follows that $nI_k$ is not a vertex of the product.
\end{proof}

\begin{corollary}\label{quotient_by_maximal_ideal_is_K}
Let $\mathfrak{m}$ be a maximal ideal of $K(\!(\mathbf{t})\!)^\circ$. 
Then the composition of the inclusion $K \hookrightarrow K(\!(\mathbf{t})\!)^\circ$ with the quotient map $ K(\!(\mathbf{t})\!)^\circ \to K(\!(\mathbf{t})\!)^\circ / \mathfrak{m}$ is an isomorphism. 
Consequently, $K(\!(\mathbf{t})\!)^\circ / \mathfrak{m} \cong K$.
\end{corollary}

\begin{proof}
Let $q \in K(\!(\mathbf{t})\!)^\circ$ and pick $\alpha_1,\ldots, \alpha_n \in K$ as in \cref{product_of_differences}. 
Then \cref{irrelevant_contained_in_maximal_ideal} implies that $\prod_{i} (q - \alpha_i) \in \mathfrak{m}$. Since $\mathfrak{m}$ is prime, at least one of the factors is in $\mathfrak{m}$. 
Thus for every $q \in K(\!(\mathbf{t})\!)^\circ$ there is an $\alpha \in K$ such that $q - \alpha \in \mathfrak{m}$. This is equivalent to the statement that the composition $K \to K(\!(\mathbf{t})\!)^\circ \to K(\!(\mathbf{t})\!)^\circ / \mathfrak{m}$ is surjective and an isomorphism as injectivity follows from the fact that $K\cap \mathfrak m=\{0\}$.
\end{proof}

Now let us show that a maximal ideal $\mathfrak{m}$ of $K(\!(\mathbf{t})\!)^\circ$ induces a monomial order on $\mathbb{N}^m$. For any $I, J \in \mathbb{N}^m$ we define
\[
I \preceq_{\mathfrak{m}} J \quad \Longleftrightarrow \quad \frac{\mathbf{t}^I}{\mathbf{t}^I + \mathbf{t}^J} \notin \mathfrak{m}.
\]

\begin{example}\label{exp_continuation}
{
We continue with Example~\ref{exp}. 
Consider a maximal ideal $\mathfrak{m}\subset K(\!(t)\!)^{\circ}$. Then either $u\prec_{\mathfrak{m}}t$ or $t\prec_{\mathfrak{m}}u$. In the first case, we have $\frac{t}{t+u}\in \mathfrak m$, so $P_w\mod \mathfrak{m}=x_{(1,1)}$. In the second case, $\frac{u}{t+u}\in \mathfrak m$ and so 
\[
\frac{t}{t+u}-1=\frac{u}{t+u}\in \mathfrak m.
\]
Hence, in this case $P_w\mod \mathfrak m=x_{(1,1)}-x_{(0,0)}$.}
\end{example}

\begin{proposition}\label{Proposition_MaxIdeal_gives_Order}
The relation $\preceq_{\mathfrak{m}}$ on $\mathbb{N}^m$ is a monomial order.
\end{proposition}

\begin{proof}
We note that the relations $I \preceq_{\mathfrak{m}} I$ and $0 \preceq_{\mathfrak{m}} J$ follow from the observations that $1/2$ and $1/(1 + \mathbf{t}^J)$ are units in $K(\!(\mathbf{t})\!)^\circ$ and therefore not in $\mathfrak{m}$.

Let $I, J \in \mathbb{N}^m$ with $I \preceq_{\mathfrak{m}} J$ and let $K\in \mathbb{N}^m$.
Since $\frac{\mathbf{t}^I}{\mathbf{t}^I + \mathbf{t}^J} \notin \mathfrak{m}$, we obtain that $$\frac{\mathbf{t}^K}{\mathbf{t}^K} \cdot \frac{\mathbf{t}^I}{\mathbf{t}^I + \mathbf{t}^J} = \frac{\mathbf{t}^{I+K}}{\mathbf{t}^{I+K} + \mathbf{t}^{J+K}} \notin \mathfrak{m}$$ and thus, $I+K \preceq_{\mathfrak{m}} J+K$.
Now let $I, J \in \mathbb{N}^m$ be arbitrary. Note that
\[
\frac{\mathbf{t}^I}{\mathbf{t}^I + \mathbf{t}^J} + \frac{\mathbf{t}^J}{\mathbf{t}^I + \mathbf{t}^J} = 1 \notin \mathfrak{m}
\]
so at least one of the two terms is not in $\mathfrak{m}$. 
So we have $I \preceq_{\mathfrak{m}} J$ or $J \preceq_{\mathfrak{m}} I$. 
Suppose that both terms are not in $\mathfrak{m}$. 
If $I \neq J$ then we have $\trop ( \mathbf{t}^J/(\mathbf{t}^I + \mathbf{t}^J) ) = \trop( \mathbf{t}^J / (\mathbf{t}^I - \mathbf{t}^J) )$, so by \cref{trop_determines_elements_of_ideal} we find that $\mathbf{t}^J / (\mathbf{t}^I - \mathbf{t}^J)$ is also not in $\mathfrak{m}$. 
Now consider the product
\[ 
q = \frac{\mathbf{t}^I}{\mathbf{t}^I + \mathbf{t}^J} \cdot \frac{\mathbf{t}^J}{\mathbf{t}^I - \mathbf{t}^J} = \frac{\mathbf{t}^{I + J}}{\mathbf{t}^{2I} - \mathbf{t}^{2J}}. 
\]
This is a product of two elements of $K(\!(\mathbf{t})\!)^\circ \setminus \mathfrak{m}$, and therefore not in $\mathfrak{m}$. 
But we have $\trop(q) \ll 1$ since $I + J$ is not a vertex of $\{2I, 2J\}$. This contradicts \cref{irrelevant_contained_in_maximal_ideal}. 
So we conclude that $I = J$, which yields the anti-symmetry.

Now it remains to show that $\preceq_{\mathfrak{m}}$ is transitive. Let $I,J,L \in \mathbb{N}^m$ all distinct with $I \preceq_{\mathfrak{m}} J$ and $J \preceq_{\mathfrak{m}} L$. 
We have 
\[ 
\frac{\mathbf{t}^{J}}{\mathbf{t}^I + \mathbf{t}^J + \mathbf{t}^L} = \frac{\mathbf{t}^J}{\mathbf{t}^I + \mathbf{t}^J}\frac{\mathbf{t}^I + \mathbf{t}^J}{\mathbf{t}^I + \mathbf{t}^J + \mathbf{t}^L} \in \mathfrak{m} 
\]
since the first factor is in $\mathfrak{m}$ and the second in $K(\!(\mathbf{t})\!)^\circ$. 
Similarly, we have
\[ 
\frac{\mathbf{t}^{L}}{\mathbf{t}^I + \mathbf{t}^J + \mathbf{t}^L} = \frac{\mathbf{t}^L}{\mathbf{t}^J + \mathbf{t}^L}\frac{\mathbf{t}^J + \mathbf{t}^L}{\mathbf{t}^I + \mathbf{t}^J + \mathbf{t}^L} \in \mathfrak{m} 
\]
Since $1 = (\mathbf{t}^I + \mathbf{t}^J + \mathbf{t}^L)/(\mathbf{t}^I + \mathbf{t}^J + \mathbf{t}^L)$ is not in $\mathfrak{m}$, we must have that $\mathbf{t}^I / (\mathbf{t}^I + \mathbf{t}^J + \mathbf{t}^L)$ is not in $\mathfrak{m}$. 
From the same factorization
\[ 
\frac{\mathbf{t}^{I}}{\mathbf{t}^I + \mathbf{t}^J + \mathbf{t}^L} = \frac{\mathbf{t}^I}{\mathbf{t}^I + \mathbf{t}^L}\frac{\mathbf{t}^I + \mathbf{t}^L}{\mathbf{t}^I + \mathbf{t}^J + \mathbf{t}^L}
\]
it then follows that $\mathbf{t}^I / (\mathbf{t}^I + \mathbf{t}^L)$ is not in $\mathfrak{m}$, so that $I \preceq_{\mathfrak{m}} L$. 
\end{proof}

{For the case $m=1$, we have that $K(\!(t)\!)^\circ=K[\![t]\!]$ and $\mathfrak m=(t)$, and we recover the fact that the unique maximal ideal encodes the unique monomial order on $\mathbb{N}$.

\begin{corollary}
\label{cor:max:ideals:unitball}
    Every maximal $k$-ideal of $V\mathbb{B}(\mathbf{t})^\circ$ is of the form 
    \begin{equation*}
        \mathfrak{m}_<^\dag=\{\tfrac{a}{b}\in V\mathbb{B}(\mathbf{t})^\circ\::\:a_{\min_<(b)}=0\}
    \end{equation*}
    for some monomial order $<$ on $\mathbb{N}^m$.
\end{corollary}
\begin{proof}
By Theorem \ref{thm_characterization} and Proposition \ref{prop_maxideal}, every maximal ideal of $K(\!(\mathbf{t})\!)^\circ$ is of the form $\mathfrak{m}_<=\{f/g \in K(\!(\mathbf{t})\!)^\circ : f_{\min\!{}_{<} ( \Supp(g) )}=0\}$ for some monomial order $<$ on $\mathbb{N}^m$.

By Theorem \ref{ideals_are_the_same} and Proposition \ref{Proposition_trop_is_Bezout}, every maximal ideal of $V\mathbb{B}(\mathbf{t})^\circ$ is the image under trop of a maximal ideal of $K(\!(\mathbf{t})\!)^\circ$, and we have $\trop(\mathfrak{m}_<)=\mathfrak{m}_<^\dag$.
\end{proof}

An important consequence is that by construction {taking initial forms commutes with} multiplication. We derive this fact from Theorem \ref{thm_characterization} as follows.

\begin{proposition}
\label{prop:multiplicative}
    Consider a pair $(w,\mathfrak{m})$ of a weight $w=(w_1,\ldots,w_n)\in\mathbb{B}[\![\mathbf{t}]\!]^n$ and a maximal ideal $\mathfrak{m}\subset K(\!(\mathbf{t})\!)^{\circ}$. If $P,Q\in F_{m,n}$, then 
    \begin{equation*}
        \text{in}_{(w,\mathfrak{m})}(PQ)=\text{in}_{(w,\mathfrak{m})}(P)\text{in}_{(w,\mathfrak{m})}(Q).
    \end{equation*}
\end{proposition}
\begin{proof}
    Let $P=\sum_Ma_ME_M$ and $Q=\sum_Nb_NE_N$, with $\trop_w(P)=\tfrac{a}{b}$ and $\trop_w(Q)=\tfrac{c}{d}$, so that $T(\trop_w(P))^{-1}=\tfrac{B}{A}$ and $T(\trop_w(Q))^{-1}=\tfrac{D}{C}$. Then by \eqref{eq:explicit_translation}, we have 
    \begin{equation*}
        \begin{aligned}
            P_wQ_w&=[\sum_M[\tfrac{B}{A}a_ME_M(\overline{\Theta(J)(e_K(w_i)}))]E_M][\sum_N[\tfrac{D}{C}b_NE_N(\overline{\Theta(J)(e_K(w_i)}))]E_N]=
            \\&=\sum_O[\sum_{M+N=O}\tfrac{BD}{AC}a_Mb_NE_M(\overline{\Theta(J)(e_K(w_i)}))E_N(\overline{\Theta(J)(e_K(w_i)}))]E_ME_N.
        \end{aligned}
    \end{equation*}

    Now let $PQ=\sum_O(\sum_{M+N=O}a_Mb_N)E_ME_N$ with $\trop_w(PQ)=\tfrac{e}{f}$, so that $T(\trop_w(PQ))^{-1}=\tfrac{F}{E}$. Then 
    \begin{equation*}
        \begin{aligned}
            (PQ)_w=\sum_O[\sum_{M+N=O}\tfrac{F}{E}a_Mb_NE_M(\overline{\Theta(J)(e_K(w_i)}))E_N(\overline{\Theta(J)(e_K(w_i)}))]E_ME_N.
        \end{aligned}
    \end{equation*}

Let $<_{\mathfrak m}$ be the monomial order obtained by applying Theorem \ref{thm_characterization}, then by Proposition \ref{prop_maxideal}, the  reduction map $\psi_{\mathfrak m}\circ\pi:K(\!(\mathbf{t})\!)^{\circ}\xrightarrow[]{}K$ appearing in \eqref{eq:reduction} can be described concretely by 
\begin{equation}
\label{eq:concrete_reduction}
    \psi_{\mathfrak m}\circ\pi(\tfrac{f}{g})=\frac{f_{\min\!{}_{<_{\mathfrak m}} ( \Supp(g) )}}{g_{\min\!{}_{<_{\mathfrak m}} ( \Supp(g) )}}.
\end{equation}

Clearly $P_wQ_w$ and $(PQ)_w$ only differ by the constants  $\tfrac{BD}{AC}$ and $\tfrac{F}{E}$.  Now we have $\frac{a}{b}\odot\frac{c}{d}=\frac{e}{f}$, since $\trop_w$ is multiplicative. In particular we have $AC=E+R$ and $BD=F+S$ where $R$ (respectively $S$) does not have any vertex of $E$ (respectively of $F$). Thus $(BD)_{\min\!{}_{<_{\mathfrak m}} ( \Supp(AC) )}=F_{\min\!{}_{<_{\mathfrak m}} ( \Supp(E) )}$ and $(AC)_{\min\!{}_{<_{\mathfrak m}} ( \Supp(AC) )}=E_{\min\!{}_{<_{\mathfrak m}} ( \Supp(E) )}$, which yields $\psi_{\mathfrak m}\circ\pi(\tfrac{BD}{AC})=\psi_{\mathfrak m}\circ\pi(\tfrac{F}{E})$. We now apply \eqref{eq:explicit_initial_deg}, which finishes the proof. 
\end{proof}

One last remark concerning the concept of initial defined on \cite{FGH20+} for polynomials with coefficients in $K[\![\mathbf{t}]\!]$ is convenient. Given a weight $w=(w_1,\ldots,w_n)\in\mathbb{B}[\![\mathbf{t}]\!]^n$ and $P=\sum_Ma_ME_M\in K[\![\mathbf{t}]\!][x_{i,J}]$, its initial was defined as \begin{equation*} 
    \In_w(P)= \sum_{\substack{\trop_w(a_{M}E_M)\cap \trop_w(P)\neq \emptyset}} \overline{a_{M}}E_M.
\end{equation*}

Given a maximal ideal $\mathfrak{m}\subset K(\!(\mathbf{t})\!)^{\circ}$ corresponding to the monomial order  $<_{\mathfrak m}$,  by \eqref{eq:concrete_reduction} we have 

 \begin{equation}
 \label{eq:explicit:fpwcoeff}
        in_{(w,\mathfrak{m})}(P)=\sum_{I+J=\min\!{}_{<_{\mathfrak m}} (\trop_w(P))\in \trop_w(a_ME_M)}[c(w)_J(a_M)_I]E_M
    \end{equation}
where $c(w)_J\in\mathbb{N}$ is a constant that depends only on the weight vector $w$. Certainly, every $\text{in}_{(w,\mathfrak{m})}(P)$ can be recovered from $\In_w(P)$.
}

It remains future work to verify whether non-maximal prime ideals of $K(\!(\mathbf{t})\!)^\circ$ can also be encoded as some type of orders of $\mathbb{N}^m$.

\section*{Acknowledgments}
This paper was finished during a research visit from S.F. and C.G. at the Unidad Oaxaca IM-UNAM in the context of a CIMPA School in Oaxaca June 2023. 
We are thankful for the financial support of the Apoyo Especial Alfonso N\'apoles G\'andara (IM-UNAM) and the great hospitality during this time. We would like to acknowledge the support and input of Mercedes Haiech in the beginning of this project.

L.B. is partially supported by PAPIIT project IA100122 dgapa UNAM. S.F. is partially supported by the grant PID2020-113192GB-I00 (Mathematical Visualization: Foundations, Algorithms and Applications) from the Spanish MICINN and by the OeAD project FR 09/2022.

\footnotesize{
\bibliographystyle{alpha}
\bibliography{bib.bib}

\newcommand{\etalchar}[1]{$^{#1}$}
\begin{thebibliography}{FLH{\etalchar{+}}23}

\bibitem[AGT16]{AGT16}
Fuensanta Aroca, Cristhian Garay, and Zeinab Toghani.
\newblock The fundamental theorem of tropical differential algebraic geometry.
\newblock {\em Pacific J. Math.}, 283(2):257--270, 2016.

\bibitem[AI16]{Aroca-Ilardi}
Fuensanta Aroca and Giovanna Ilardi.
\newblock Newton's lemma for differential equations.
\newblock {\em Illinois J. Math.}, 60(3-4):859--867, 2016.

\bibitem[AR19]{ArocaRond}
Fuensanta Aroca and Guillaume Rond.
\newblock Support of {L}aurent series algebraic over the field of formal power
  series.
\newblock {\em Proc. Lond. Math. Soc. (3)}, 118(3):577--605, 2019.

\bibitem[BBG23]{BG}
F\'elix Baril~Boudreau and Cristhian Garay.
\newblock Tropicalization of schemes and sheaves.
\newblock {\em arXiv:2304.04872 [math.AG]}, 2023.

\bibitem[BG06]{bazzoni2006prufer}
Silvana Bazzoni and Sarah Glaz.
\newblock Pr{\"u}fer rings.
\newblock In {\em Multiplicative Ideal Theory in Commutative Algebra: A Tribute
  to the Work of Robert Gilmer}, pages 55--72. Springer, 2006.

\bibitem[Bos21]{Bos_initial}
Lara Bossinger.
\newblock Full-rank valuations and toric initial ideals.
\newblock {\em Int. Math. Res. Not. IMRN}, (10):7433--7469, 2021.

\bibitem[CGL23]{CGL}
Ethan Cotterill, Cristhian Garay, and Johana Luviano.
\newblock Exploring tropical differential equations.
\newblock {\em arXiv:2012.14067, to appear in Advances in Geometry}, 2023.

\bibitem[EH00]{EH_GeomSch}
David Eisenbud and Joe Harris.
\newblock {\em The geometry of schemes}, volume 197 of {\em Graduate Texts in
  Mathematics}.
\newblock Springer-Verlag, New York, 2000.

\bibitem[Eis95]{Eisenbud_commutative}
David Eisenbud.
\newblock {\em Commutative algebra}, volume 150 of {\em Graduate Texts in
  Mathematics}.
\newblock Springer-Verlag, New York, 1995.
\newblock With a view toward algebraic geometry.

\bibitem[Fer97]{ferrero1997prime}
Miguel Ferrero.
\newblock Prime ideals in polynomial rings in several indeterminates.
\newblock {\em Proceedings of the American Mathematical Society},
  125(1):67--74, 1997.

\bibitem[FLH{\etalchar{+}}20]{FGH20}
S.~Falkensteiner, C.~Garay L\'opez, M.~Haiech, M.P. Noordman, Z.~Toghani, and
  F.~Boulier.
\newblock The fundamental theorem of tropical partial differential algebraic
  geometry.
\newblock {\em Proc. 45th ISSAC (2020)}, 2020.

\bibitem[FLH{\etalchar{+}}23]{FGH20+}
S.~Falkensteiner, C.~Garay L\'opez, M.~Haiech, M.P. Noordman, F.~Boulier, and
  Z.~Toghani.
\newblock On initials and the fundamental theorem of tropical partial
  differential geometry.
\newblock {\em Journal of Symbolic Computation}, 115:53--73, 2023.

\bibitem[FT22]{FT20}
Alex Fink and Zeinab Toghani.
\newblock Initial forms and a notion of basis for tropical differential
  equations.
\newblock {\em Pacific J. Math.}, 318(2):453--468, 2022.

\bibitem[GM23]{mereta_fund}
J.~Giansiracusa and S.~Mereta.
\newblock A general framework for tropical differential equations.
\newblock {\em manuscripta math.}, 2023.

\bibitem[Gub13]{gubler2013guide}
Walter Gubler.
\newblock A guide to tropicalizations.
\newblock {\em Algebraic and combinatorial aspects of tropical geometry},
  589:125--189, 2013.

\bibitem[HG20]{HuGao2020}
Youren Hu and Xiao-Shan Gao.
\newblock {Tropical Differential Gr\"{o}bner Bases}.
\newblock {\em {Mathematics in Computer Science}}, 2020.

\bibitem[HH11]{HerzogHibi}
J\"{u}rgen Herzog and Takayuki Hibi.
\newblock {\em Monomial ideals}, volume 260 of {\em Graduate Texts in
  Mathematics}.
\newblock Springer-Verlag London, Ltd., London, 2011.

\bibitem[MS15]{maclagan2015introduction}
Diane Maclagan and Bernd Sturmfels.
\newblock {\em Introduction to tropical geometry}, volume 161 of {\em Graduate
  Studies in Mathematics}.
\newblock American Mathematical Society, Providence, RI, 2015.

\bibitem[RY08]{RY}
Wolfgang Rump and Yi~Chuan Yang.
\newblock Jaffard-{O}hm correspondence and {H}ochster duality.
\newblock {\em Bull. London Math. Soc.}, 4:263 -- 273, 2008.

\end{thebibliography}
}

\end{document}